\documentclass[11pt]{amsart}
\usepackage{amsmath}
\usepackage{epsfig}
\usepackage{amsthm}
\usepackage{amssymb}
\usepackage{euscript}

\DeclareFontFamily{U}{mathx}{\hyphenchar\font45}
\DeclareFontShape{U}{mathx}{m}{n}{
      <5> <6> <7> <8> <9> <10>
      <10.95> <12> <14.4> <17.28> <20.74> <24.88>
      mathx10
      }{}
\DeclareSymbolFont{mathx}{U}{mathx}{m}{n}
\DeclareMathAccent{\widecheck}    {0}{mathx}{"71}

\theoremstyle{plain}
\newtheorem{theorem}{Theorem}[section]
\newtheorem{lemma}[theorem]{Lemma}

\newtheorem{cor}[theorem]{Corollary}

\theoremstyle{definition}
\newtheorem{remark}[theorem]{Remark}
\newtheorem{dfn}[theorem]{Definition}

\numberwithin{equation}{section}
\usepackage{amsmath, amsthm, amssymb, latexsym, mathrsfs}

\begin{document}

\title{Compactness of Alexandrov-Nirenberg Surfaces}
\author[Han]{Qing Han}
\address{Department of Mathematics\\
University of Notre Dame\\
Notre Dame, IN 46556, USA} \email{qhan@nd.edu}
\address{Beijing International Center for Mathematical Research\\
Peking University\\
Beijing, 100871, China} \email{qhan@math.pku.edu.cn}
\author[Hong]{Jiaxing Hong}
\address{School of Mathematical Science, Fudan University, Shanghai, 200433, China}
\email{jxhong@fudan.edu.cn}
\author[Huang]{Genggeng Huang }
\address{School of Mathematical Science, Fudan University, Shanghai, 200433, China}
\email{genggenghuang@fudan.edu.cn}
\thanks{The first author acknowledges the support of NSF
Grant DMS-1105321.
The second and the third authors acknowledge the support of NNSF Grants 11121101 and 11131005.}
\begin{abstract}
We study a class of compact surfaces in $\mathbb R^3$ introduced by
Alexandrov and generalized by
Nirenberg and prove a compactness result under suitable assumptions
on induced metrics and Gauss curvatures.
\end{abstract}
\maketitle

\section{Introduction}

Compactness plays an important role in many subjects of mathematics. With the compactness
of certain classes of geometric or analytic objects, one can take limits in appropriate topology and,
by analyzing the limits, obtain desired properties of the entire classes.
Nirenberg \cite{N1} proved that the class of smooth closed convex surfaces in $\mathbb R^3$ is compact
in $C^k$-topology, for any positive integer $k\ge 3$.
This result plays an important role in his solution of the
Weyl problem, concerning the isometric embedding of smooth metrics on $\mathbb S^2$ in
$\mathbb R^3$.

Isometric embedding is a classical problem in differential geometry.
In 1916, Weyl \cite{Weyl1916} studied whether every smooth metric
on $\mathbb S^2$ with positive Gauss curvature admits a
smooth isometric embedding in $\mathbb R^3$. This problem, now referred to as the
Weyl problem, was solved by Nirenberg \cite{N1}
and   Pogorelov \cite{Pogorelov1953}
independently in the early
1950s.
In 1990s, Guan and Li \cite{GL}, and Hong and Zuily \cite{HZ} independently
generalized this result to metrics on $\mathbb S^2$ with
nonnegative Gauss curvature.

Closely related to the global isometric embedding problem is the
rigidity question. The first rigidity result was proved by
Cohn-Vossen \cite{Cohn-Vossen1927} in 1927; this states that any two closed isometric
analytic convex surfaces are congruent to each other. In 1943, Herglotz \cite{Herglotz1943} gave
a short proof of the rigidity, assuming that the surfaces are
three times continuously differentiable. In 1962 it was
extended by Sacksteder \cite{Sacksteder1962} to surfaces  with no more than two times
continuously differentiable metrics.

It is natural to study the isometric embedding and the rigidity  for surfaces with
Gauss curvature of mixed sign. For rigidity, Alexandrov \cite{AL} in 1938 introduced a class of
surfaces satisfying some integral condition for its Gauss curvature
and proved that any compact analytic surface with this condition is
rigid. In 1963, Nirenberg \cite{N2} generalized this result to smooth
surfaces under extra assumptions.

For the global isometric embedding of metrics defined on general compact surfaces,
Han and Lin
\cite{HL} recently made the first attempt and discussed the isometric embedding of metrics
defined on torus. They found obstructions to the existence of
such isometric embedding. Specifically, they found a one-parameter
family of analytic metrics which are small perturbations of the
standard metric on torus and do not admit any $C^2$ isometric
embedding in $\mathbb R^3$.

We should point out that vanishing Gauss curvature causes serious problems
even for the local isometric embedding of 2-dimensional
Riemannian metrics in $\mathbb R^3$. In 1985 and 1986, Lin \cite{Lin1985},
\cite{Lin1986} proved the existence of
sufficiently smooth local isometric embedding if the Gauss curvature is nonnegative or
the Gauss curvature changes sign cleanly. For other results on local isometric
embedding, refer to \cite{Han2005}, \cite{Han2005b}, \cite{Han-Hong-Lin2003} and
\cite{Han&Khuri2010}.

As is well-known, a closed surface $M$ in $\mathbb R^3$ satisfies
\begin{equation}\label{eq-ANbigger4pi}\int_{M} K^+dg\ge 4\pi,\end{equation}
where $K$ is the Gauss curvature of $M$ and $K^+$ is its positive part,
i.e., $K^+=\max\{0, K\}$.
This simply says that the image of the Gauss map on $\{p\in M:\, K>0\}$ covers
the unit sphere $\mathbb S^2$ at least once. Such an integral condition provides
an obstruction for the existence of isometric embedding of metrics
on closed surfaces.

Alexandrov \cite{AL} and Nirenberg \cite{N2} studied oriented
closed surfaces in $\mathbb R^3$ satisfying the equality in \eqref{eq-ANbigger4pi} and
proved the rigidity under appropriate
non-degeneracy condition for the Gauss curvature.
Since \eqref{eq-ANbigger4pi} involves the part of the surface where
Gauss curvature is positive, we will formulate results by Alexandrov and Nirenberg
accordingly as follows. Refer to \cite{N2}, or \cite{HQH}, for a proof.

\smallskip

\noindent
{\bf Theorem A.} {\it Let $\Sigma$ be an oriented and bounded
$C^4$-surface in $\mathbb R^3$ with nonempty boundary. Suppose}
\begin{align}\label{eq-ANCondition} \begin{split}
&K>0\quad\text{in }\Sigma,\\
&K=0\text{ and }\nabla K\neq 0\quad\text{on }\partial\Sigma, \\
&\int_{\Sigma}Kdg=4\pi.\end{split}\end{align}
{\it Then,

$\operatorname{(1)}$ $\partial\Sigma$ consists
of finitely many smooth planar convex curves $\sigma_j$,  $j=1,...,J$. Moreover,
the plane containing $\sigma_j$ is tangent to $\Sigma$ along $\sigma_j$, for each $j=1,\cdots,J$;

$\operatorname{(2)}$ the geodesic curvature $k_g$ of $\sigma_j$ is negative, for each $j=1,\cdots,J$;

$\operatorname{(3)}$ $\Sigma\bigcup \partial\Sigma$ is rigid.}

\smallskip

In this paper, we initiate our study of surfaces introduced by Alexandrov and Nirenberg, as in
Theorem A. For convenience, we introduce the following terminology.

\begin{dfn}\label{dfn-ANSurfaces} We call $\Sigma$ an
{\it Alexandrov-Nirenberg surface} if it satisfies \eqref{eq-ANCondition}. \end{dfn}

Our ultimate goal is to study the isometric embedding related to Alexandrov-Nirenberg
surfaces. The rigidity result in Theorem A(3) can be interpreted as the uniqueness of the
isometric embedding. We are interested in the existence of the related isometric embedding.
If we attempt to employ the method of continuity to
prove the existence of such an isometric embedding, a necessary step is to prove the
closedness of the embeddable metrics, or the compactness of the associated surfaces.

The main result of the present paper is the following compactness result.

\begin{theorem}\label{them-main} For each integers $J\ge 1$ and $k\ge 2$,
a constant $\alpha\in (0,1)$  and a positive constant
$C$, let $\mathcal{S}_{J, k,\alpha, C}$ be the collection of Alexandrov-Nirenberg surfaces $\Sigma$
of class $C^{k+3, \alpha}$, with
$J$ connected components in $\partial\Sigma$, such that
$$|g|_{C^{k+2, \alpha}(\bar\Sigma)}+ \max_{\partial \Sigma}\frac
1{|\nabla K|}+ \max_{\partial \Sigma}\frac 1{|k_g|}\le C,$$
where $g$ is the induced metric on $\Sigma$, $K$ is the Gauss curvature of $\Sigma$ and
$k_g$ is the geodesic curvature of $\partial\Sigma$. Then, $\mathcal{S}_{J, k,\alpha,C}$ is compact
in the $C^k$-topology.
\end{theorem}

We note that $\nabla K$ does not vanish on $\partial\Sigma$ by \eqref{eq-ANCondition}
and that $k_g$ does not vanish on $\partial\Sigma$ by Theorem A(2).

Theorem \ref{them-main} is based on a priori estimates of the $C^{k, \alpha}$-norms of the position
vectors of $\Sigma$ in $\mathbb R^3$ or its associated second
fundamental form.
Difficulties in deriving such estimates arise from the condition that
curvature $K=0$ on $\partial \Sigma$. As is well known, vanishing
Gauss curvature results in degeneracy of the associated
nonlinear elliptic equations. In \cite{H1} and \cite{HHW}, Hong
studied the case where $\partial\Sigma$ consists of one connected component
and the geodesic curvature $k_g$ of $\partial \Sigma$ is positive
everywhere. For more bibliography, see \cite{HQH}. However, in our
present case, $k_g<0$ on $\partial \Sigma$ by Theorem A(2). From an
analytic point of view, the associated elliptic equation is
non-characteristically degenerate on $\partial \Sigma$ if $k_g>0$ on
$\partial \Sigma$ and is characteristically degenerate if $k_g<0$.
The latter is presumably more difficult to study than the former.
For the characteristically degenerate elliptic equations in this
paper, the usual barrier arguments do not seem to work for the
estimate of the difference-quotient along the normal to the
degenerate boundary, although derivatives of solutions on
the boundary can be solved from the equation. This is  the major
difficulty we encounter in the present paper.

To prove Theorem \ref{them-main},  we need to derive a priori
estimates of the second fundamental forms.
The crucial part is the estimate of the boundary Lipschitz norm.
We achieve this in three successive steps:

Step 1. Estimate the $L^\infty$-norm by the maximum principle;

Step 2. Estimate the boundary H\"{o}lder norm by de Giorgi
iteration;

Step 3. Estimate the boundary Lipschitz norm by blowup
arguments.

The method used in Step 2 and Step 3
is of independent interest.

After deriving the boundary Lipschitz norm of the second fundamental
form, we obtain estimates of the boundary higher order
norm by results in \cite {HH} on $L^p$ and H\"{o}lder
boundary estimates for a class of characteristically degenerate elliptic equations.

The paper is organized as follows.
In Section \ref{sec-InteriorEstimates},
we derive a global upper bound of the mean curvature and an interior estimate  of
higher order derivatives of position vectors for Alexandrov-Nirenberg surfaces.
In Section \ref{sec-GaussCodazziEqs}, we derive some important equations
in geodesic coordinates near boundary.
In Section \ref{sec-BoundsMeanCurv}, we derive upper and lower bounds of the mean curvature by
the maximum principle. We derive boundary H\"{o}lder norms of the second fundamental form
by de Giorgi iteration in Section \ref{sec-BoundaryHolder} and boundary Lipschitz norms by blowup
arguments in Section \ref{sec-BoundaryLipschitz}. In Section \ref{sec-HigherOrderEstimates}, we
provide an estimate of higher order derivatives of the second fundamental form
and prove Theorem \ref{them-main}. Section
\ref{sec-Appendix} is an appendix, where we reformulate results in \cite {HH}
for our applications.

\section{Interior Estimates}\label{sec-InteriorEstimates}

In this section, we derive a global upper bound of the mean curvature and an interior estimate
of higher order derivatives of position vectors for Alexandrov-Nirenberg surfaces.

Suppose $\Sigma$ is an Alexandrov-Nirenberg surface as introduced in
Definition \ref{dfn-ANSurfaces}.
By Theorem A, $\partial\Sigma$ consists of finitely many planar
convex curves. Let $\sigma$ be a connected component in $\partial\Sigma$.
Without loss of generality, we assume that, in the
geodesic coordinates with the base curve
$\sigma,$ the induced metric $g$ is  of the form
\begin{equation}\label{geod1}g=B^2ds^2+dt^2\quad\text{for any } (s,t)
\in [0 ,2\pi ]\times [0,1],\end{equation}
where
$B$ is a positive
function in $[0 ,2\pi ]\times [0,1]$ satisfying
\begin{equation}\label{geod2}B(\cdot,0)=1,\quad
B_t(\cdot,0)=-k_g.\end{equation}
Here, $t=0$ corresponds to the boundary curve $\sigma$ and the
negative sign in $B_t$ indicates that
the geodesic curvature of $\sigma$ is calculated with respect
to the anticlockwise orientation.
Obviously, we have $B_t>0$ on
$\sigma$.
Furthermore, we assume, by a scaling in $t$ if necessary, that
$$B_t>0\ \text { for all }t\in
[0,1].$$
Throughout the paper, we adopt the notion
$(\partial_s,\partial_t)=(\partial_1,\partial_2)$. It is easy to
calculate
\begin{align*}
&\Gamma^1_{11}=\frac {B_s}B,\quad\ \ \ \Gamma^1_{12}=\frac {B_t}B,\
\ \Gamma^1_{22}=0,\\
\quad &\Gamma^2_{ 11}=-BB_t,\ \ \Gamma^2_{12}=0, \ \ \ \
\Gamma^2_{22}=0.\end{align*} The Gauss-Codazzi equations are given by
\begin{align}\label{GC1}L_t-M_s&=\frac{B_t}{B}L-\frac{B_s}{B}M+BB_tN,\\
\label{GC2}M_t-N_s&=-\frac{B_t}{B}M,\end{align}
and
\begin{equation}\label{GC3}NL-M^2=KB^2.\end{equation} Note that the
mean curvature $H$
is given by
\begin{equation}\label{mean}H=\frac12\left(\frac{L}{B^2}+N\right).
\end{equation}

\begin{lemma}\label{Lemma1.1} Let $\Sigma$ be an Alexandrov-Nirenberg surface
in $\mathbb R^3$ of class $C^4$
and $\sigma$ be a connected component in $\partial\Sigma$.
Then, in the geodesic
coordinates as in \eqref{geod1} and \eqref{geod2},
$$L=M=0,\quad N=\sqrt {\frac {K_t}{B_t}}\quad\text{on }t=0,$$
and
$$L_t=\sqrt {K_tB_t}\quad\text{on }t=0.$$
\end{lemma}

\begin{proof} We first note $K=0$ and $L=M=0$ as $t=0$ since the normal
to $\Sigma$ is constant along
$\sigma$. By (\ref{GC1}) and $t$-differentiation of (\ref{GC3}),
we have
$$L_t=NB_t, \quad NL_t=K_t\quad \text{at }t=0.$$
Solving the above system yields expressions of $N$ and $L_t$ on $t=0$. \end{proof}

By Lemma \ref{Lemma1.1}, $L$, $M$, $N$ and $L_t$ are
intrinsically determined on $\sigma$.

\begin{cor}\label{Cor1.2} Let $k$ be a nonnegative integer,
$\Sigma$ be an Alexandrov-Nirenberg surface
in $\mathbb R^3$ of class $C^{k+4}$
and $\sigma$ be a connected component in $\partial\Sigma$.
Then, in the geodesic
coordinates as in \eqref{geod1} and \eqref{geod2},
all $k$-derivatives of $M$ and $N$ and
$(k+1)$-derivatives of $L$ are intrinsically determined on
$\sigma$.
\end{cor}

\begin{proof} We will prove by induction that, for any integer $m= 0, 1, \cdots, k$,
$$\partial^mL, \partial^mM, \partial^mN\text{ and }\partial^{m+1}_tL
\text{ are intrinsically determined on }\sigma.$$ Here
$\partial^m$ denotes all derivatives of order $m$. Note that $m=0$
corresponds to Lemma \ref{Lemma1.1}. We assume it holds for $m-1$
and consider $m$ for some $m\ge 1$. Since pure $s$-derivatives of
$\partial^{m-1}L, \partial^{m-1}M, \partial^{m-1}N$ and $\partial^{m}_tL$
are intrinsically determined on
$\sigma$, we consider only $\partial^{m+1}_tL$, $\partial_t^mM$
and $\partial_t^mN$. First,  $\partial^m_tM$ is intrinsically
determined on $\sigma$ by (\ref{GC2}). Next,  a differentiation
of (\ref{GC1}) and (\ref{GC3}) with respect to $t$ of appropriate
order yields
$$\partial^{m+1}_{t}L-BB_t\partial^m_tN=\cdots, \quad
N\partial^{m+1}_{t}L+m\partial_tL\partial^m_tN=\cdots,$$ where
$\cdots$ indicates expressions intrinsically determined on $\sigma$.
Here we used $L=M=0$ on $\sigma$. The coefficient matrix at
$t=0$ is given by
$$\left(\begin{matrix} 1& -B_t\\
\sqrt{\frac{K_t}{B_t}}&m\sqrt{K_tB_t}\end{matrix}\right),$$ which is
nonsigular. Hence, $\partial^{m+1}_{t}L$ and $\partial^m_tN$ are
intrinsically determined on $\sigma$. \end{proof}

\begin{cor}\label{Cor1.3} Let $k$ be a nonnegative integer,
$\Sigma$ be an Alexandrov-Nirenberg surface
in $\mathbb R^3$ of class $C^{k+4}$
and $\sigma$ be a connected component in $\partial\Sigma$.
Then, in the geodesic
coordinates as in \eqref{geod1} and \eqref{geod2},
all $k$-derivatives of the mean curvature $H$ are intrinsically determined on
$\sigma$.
\end{cor}

Next, for the  Alexandrov-Nirenberg
surface $\Sigma$ in $\mathbb R^3$,
we assume by Theorem A(1) that
$\partial \Sigma$ consists of $J$  planar
convex curves. Hence, $\Sigma$ and the planar convex regions
enclosed by these curves form a convex surface $\widetilde\Sigma$ in $\mathbb R^3$.
We first have the following result.

\begin{lemma}\label{inball}
There exists a ball of radius $R_0$ inside $\widetilde \Sigma$, where
$R_0$ is a positive constant depending only on $1/\max
K$ and the intrinsic diameter $l$ of $\Sigma$.
\end{lemma}

For a proof, refer to \cite {CY}  or
\cite{HQH} p196.
In the following, we
always take the origin as the center of this ball.

We fix a bounded domain $D\subset \mathbb R^2$ with $J$ connected components in
$\partial D$. Then, the induced metric $g$
of $\bar \Sigma$ can be viewed as a metric in $\bar D$ and $\bar \Sigma$
an isometric embedding of $(\bar D, g)$.
Let ${\bf r}$ be the position vector of $\Sigma$ in $\mathbb R^3$ and set
$$\rho=-\frac12|{\bf r}|^2.$$
We always regard ${\bf r}$ and all related
functions as defined in $\bar D$.

\begin{lemma}\label{Lemma1.10} There hold
\begin{equation}\label{eq-Position1}Kh^{ij}\rho_{ij}=-2H-2K{\bf r\cdot n},\end{equation}
and
\begin{equation}\label{eq-Position2}({\bf r\cdot n})^2+|\nabla \rho|^2=|{\bf r}|^2.\end{equation}
\end{lemma}

\begin{proof} For any $p\in D$, we take the orthonormal coordinates at $p$.
Then a straightforward calculation yields
\begin{equation}\label{1.1}\rho_{ij}=-{\bf r}_i\cdot {\bf r}_j-{\bf r}_{ij}\cdot
{\bf r}=-\delta_{ij}-h_{ij}{\bf r\cdot n},\end{equation} and
$$Kh^{ij}\rho_{ij}=-Kh^{ij}\delta_{ij}-Kh^{ij}h_{ij}{\bf r\cdot n}=-2H-2K{\bf r\cdot n}.$$
This proves the first part. As for the second part, we have
\begin{align*}|{\bf r}|^2-({\bf r\cdot n})^2&=|{\bf r\times
n}|^2=\left|{\bf r}\times \frac{{\bf r}_1\times {\bf r}_2}{|{\bf
r}_1\times
{\bf r}_2|}\right|^2\\
&=\left|\frac{({\bf r\cdot r}_1){\bf r}_2-({\bf r\cdot r}_2){\bf
r}_1}{|{\bf r}_1\times {\bf
r}_2|}\right|^2=g^{ij}\rho_i\rho_j.\end{align*} This finishes the
proof.\end{proof}

Now we prove an upper bound of the mean curvature.

\begin{lemma}\label{Lemma1.11} Let $\Sigma$ be an Alexandrov-Nirenberg
surface in $\mathbb R^3$ of class $C^5$. Then,
$$H\le C\left\{\max_{\partial \Sigma}\sqrt {\frac {|\nabla K|}{|k_g|}}
+\max_{\Sigma}K +\max_{\Sigma}\sqrt {|\Delta K|}\right\},$$ where $C$ is a
positive constant depending only on the intrinsic diameter of
$\Sigma$.
\end{lemma}

Similar estimates were obtained in \cite{N1} and
\cite{Y2} for closed surfaces with positive Gauss curvature and in
\cite{GL} and \cite{HZ} for closed surfaces with nonnegative Gauss curvature.
Lemma \ref{Lemma1.11} extends these results to surfaces with boundary,
where Gauss curvature vanishes.

\begin{proof} By our convention, the induced metric $g$ and all related functions
are defined in $\bar D\subset\mathbb R^2$.
First, we
recall a differential equation satisfied by $H$.
For any $p\in D$, we take the
orthonormal coordinates at $p$ and then have
\begin{equation}\label{eqH}Kh^{ij}H_{ij}=(h_{12,l}h_{12,l}-h_{11,l}h_{22,l})
+2KH^2+\frac
12( \Delta K-4K^2)\quad\text {at }p,\end{equation} where $h_{ij}$
is the coefficient of the second fundamental form of
${\bf r}$, $i,j=1,2$, and $(h^{ij})$ is the inverse matrix of $(h_{ij})$. See
\cite{N1}, \cite {Y2} or \cite{HQH} p182 for a derivation.
Set $w=He^{-\lambda \rho}$ for a constant
$\lambda$ to be fixed. Then $w$ satisfies
\begin{align}e^{\lambda
\rho}Kh^{ij}&w_{ij}+2\lambda
e^{\lambda\rho}Kh^{ij}\rho_iw_j=(h_{12,l}h_{12,l}-h_{11,l}h_{22,l})\nonumber\\
&+2K H^2+\frac 12(\Delta K-4K^2) -\lambda
HKh^{ij}\rho_{ij}-\lambda^2
HKh^{ij}\rho_i\rho_j,\label{eqH1}\end{align} at any point where
the orthonormal coordinates are taken.

Suppose $w$ attains its maximum over $\bar D$ at some point $p$. If
$p\in\partial D,$ then Lemma \ref{Lemma1.1} yields
\begin{equation}\label{H1}
w=\frac 12e^{-\lambda \rho}\left(N+\frac{L}{B^2}\right)= \frac
12e^{-\lambda \rho}\sqrt {\frac{K_t}{B_t}}.
\end{equation}
If $p\in D$, we take the orthonormal
coordinates at $p$ and then have
\begin{equation}\label{H2}
w_i=0\text{ and }Kh^{ij}w_{ij}\le 0\text{ at }p.
\end{equation}
Without loss of generality we may assume $h_{12}(p)=0$.
Consequently, we obtain at $p$
$$0=e^{\lambda\rho}w_l=H_l-\lambda H\rho_l=\frac 12(h_{11,l}+h_{22,l})-\lambda H\rho_l,$$
and hence
\begin{align*}h_{12,l}h_{12,l}-h_{11,l}h_{22,l}&=\sum_{l=1}^2\big((h_{12,l})^2
+(h_{1
1,l})^2\big)-2\lambda h_{11,l}H\rho_l\\
=&\sum_{l=1}^2\big((h_{12,l})^2+\left(h_{11,l}-\lambda
H\rho_l\right)^2\big)-\lambda^ 2H^2|\nabla \rho|^2.\end{align*}
By \eqref{eq-Position1} and \eqref{eq-Position2},  we have
\begin{align*}&\ -\lambda
HKh^{ij}\rho_{ij}-\lambda^2 HKh^{ij}\rho_i\rho_j\\
=&\ 2\lambda H^2+2\lambda HK{\bf r\cdot n}-\lambda^2 HKh^{ij}\rho_i\rho_j\\
\ge&\  2\lambda H^2-4\lambda^2H^2|\nabla \rho|^2-2\lambda HK|{\bf
r}|.\end{align*} Thus inserting all the above estimates into
(\ref{eqH1}) with the aid of (\ref{H2}) yields, at $p$,
$$\frac 12(4K^2-\Delta K)
\ge 2(\lambda -3|\nabla \rho|^2\lambda^2)H^2-2\lambda HK|{\bf
r}|.$$ Let $l$ be the intrinsic diameter of $\bar D$ in $g$. Then,
\eqref{eq-Position2} implies $|\nabla \rho|\le |{\bf r}|\le l$. Therefore,
we get at $p$
$$\frac 14(4K^2-\Delta K)
\ge \lambda(1 -3\lambda l^2)H^2-\lambda lHK.$$ Choosing $\lambda
=1/{4l^2}$, we have  at $p$
$$4l^2(4K^2-\Delta K)\ge H^2-4lHK.$$
This yields at $p$
$$H\le 10l\big(K+\sqrt{|\Delta K|}\big),$$
or $$w\le 10le^{-\lambda\rho}\big(K+\sqrt{|\Delta K|}\big).$$ This
yields the desired result.
\end{proof}

Next, we derive interior estimates of higher derivatives of position vectors.
Heinz \cite{H} derived such interior estimates if $D$ is a disk, namely,
if $\partial D$ consists of one connected component. Next, we provide
a direct proof for interior estimates in the general setting.

Set
$$\rho=-\frac 12|{\bf r}|^2.$$
By  (\ref{1.1}),   we have
\begin{equation}\label{rho1}
\det (\rho_{ij}+\delta_{ij})=K(-2\rho-|\nabla \rho|^2),
\end{equation} and
\begin{equation}\label{rho2}
-2\rho-|\nabla \rho|^2=|-{\bf r\cdot n}|^2\ge R^2_0,
\end{equation}
by the geometric meaning of the Minkowski function
$-{\bf r\cdot n}$ and Lemma \ref{inball}.

\begin{theorem}\label{inregularity}
Let $\Sigma$ be an Alexandrov-Nirenberg surface of class $C^5$ in $\mathbb R^3$
with the
induced metric $g$ defined on $\bar D$ and let ${\bf r}$ be the position vector of $\Sigma$.
Then for any subdomain $D'\subset\subset D_1\subset\subset D$, the principal curvatures
$k_i$ satisfy
\begin{equation}\label{eq-PrincipalLowerBound}
k_i\ge \frac 1C \quad\text {in } D',\end{equation}
where $C$ is a positive constant depending only on
\begin{equation}\label{eq-PrincipalLowerBoundDependence}
|g|_{C^4(\bar D)},\,\, \frac 1{\min_{\partial D}k_g},\,\, \frac 1{\min_{D'}K}.\end{equation}
Moreover,
there exists a constant $\alpha\in (0,1)$ depending only on the quantities
in (\ref{eq-PrincipalLowerBoundDependence}) such that,
 if $\Sigma$ is $C^{k+3}$, for some $k\ge 2$, then
$$|{D^2\bf r}|_{C^{k,\alpha}(\overline D')}
\le C,$$
where $C$ is a positive constant depending only on
$$k,\,\, |g|_{C^{k+2}(\bar D)},\,\, \frac 1{\min_{\partial D}k_g},\,\, \frac 1{\inf_{D_1}K},
\,\,\frac 1{\operatorname{dist}(D',\partial D_1)}.$$
\end{theorem}

\begin{proof} Let $k_1$ and $k_2$ be principal curvatures.
First we note
$$k_i\le 2H.$$
Hence, $k_1$ and $k_2$ are bounded from above by Lemma \ref{Lemma1.11}. Moreover,
\begin{equation}\label{H4}
k_1=\frac K{k_2}\ge\frac 1{2\max H}\inf_{\bar {D}'}K\ge\frac 1{C_{D'}}. \end{equation}
Hence, $k_1$ and $k_2$ have a positive lower bound in $D'$.
In particular, the second fundamental
form has a positive lower bound in $D'$.
 Suppose that $D'\subset\subset D_1\subset\subset D$.
 Then  the second fundamental form $h_{ij}$ of the given surface $\bf{r}$ satisfies
\begin{equation}\label{H5}
\frac 1{C_*}I\le \big(h_{ij}\big)\le C_*I \quad\text {in }D_1,
\end{equation} where $C_*$ is a positive constant depending only
\begin{equation}\label{H6}
|g|_{C^4(\bar D)},\quad \frac 1{\min_{\partial D}k_g},\quad \frac 1{\min_{D_1}K}.
\end{equation}

In view of
(\ref{1.1}) again, it follows that
\begin{equation}\label{rho3}
\nabla_{ij}\rho=h_{ij}\sqrt{-2\rho-|\nabla \rho|^2}-g_{ij}
\end{equation}are bounded in $D$.
Note
$$\frac {\partial}{\partial
(\nabla_{ij}\rho)}\left(\det\big(\nabla_{ij}\rho+\delta_{ij}\big)\right)=Kh^{ij}
(-2\rho-|\nabla \rho|^2).$$
A covariant differentiation of (\ref{rho1}) yields,
for $l=1,2$,
$$K(-2\rho-|\nabla \rho|^2)h^{ij}\nabla_l\nabla_{ij}\rho
=\nabla_l\big(K(-2\rho-|\nabla \rho|^2)\big).$$
Hence by the Ricci, identity we have, for $l=1,2$,
\begin{equation}\label{H7}
h^{ij}\nabla_{ij}\rho_l=
\frac {\nabla_l\big(K(-2\rho-|\nabla \rho|^2)\big)}{K(-2\rho-|\nabla \rho|^2)}-h^{ij}\rho_mR_{mlij}
\equiv g_l \quad\text {in }D_1.
\end{equation}
Obviously, $g_l$, $l=1,2$  are bounded in $D_1$ by (\ref{H5}) and (\ref{rho2}).
Moreover, (\ref{H5}) implies that  (\ref{H7}) is uniformly elliptic
with bounded coefficients in two dimensional space.
Now by Theorem 12.4 in \cite{GT}, we have, for $i=1,2$,
$$[\rho_l]_{C^{1,\alpha}(D')}\le C(1+|g_l|_{L^{\infty}(D_1)})\le C_1, $$
where $\alpha=\alpha(C_*)\in (0,1)$ and $C, C_1$ are positive constants
depending only on the quantities in (\ref{H6}).
Therefore combining the above inequalities with the  structure equations
$${\bf r}_{ij}=\Gamma_{ij}^k{\bf r}_k+h_{ij}{\bf n}=\Gamma_{ij}^k{\bf r}_k
+\frac {\nabla_{ij}\rho+g_{ij}}{\sqrt{-2\rho-|\nabla \rho|^2}}{\bf n},$$ we obtain
$$|D^2{\bf r}|_{C^\alpha(D')}\le C,$$
where $C$ is a positive constant depending only on the quantities in (\ref{H6}).

Combining the standard regular theory of
elliptic equations with the structure equations, we have, for $k\ge 2$,
$$
|D^k\rho|_{C^{\alpha}(\overline D')},+|D^k\bf{r}|_{C^{\alpha}(\overline D')}\le
\le C,$$
where $C$ is a positive constant depending only on
$$k,\,\, |g|_{C^{k+2}(\bar D)},\,\, \frac 1{\min_{\partial D}k_g},\,\, \frac 1{\inf_{D_1}K},
\,\, \frac 1{\operatorname{dist}(D',\partial D_1)}.
$$
This is the desired estimate.
\end{proof}

\section{Gauss-Codazzi Equations near Boundary}\label{sec-GaussCodazziEqs}

Suppose $\Sigma$ is an Alexandrov-Nirenberg surface as introduced in
Definition \ref{dfn-ANSurfaces}. The primary goal in this
paper is to derive estimates of the second fundamental form near the
boundary $\partial\Sigma$.

Let $\sigma$ be a connected component in $\partial\Sigma$ and
$L, M$ and $N$ be the coefficients of the second fundamental form
in the geodesic
coordinates as in \eqref{geod1} and \eqref{geod2} near $\sigma$.
In this section, we derive differential equations of $1/N$ and $M$.
We first
derive an equation of $1/N$.

\begin{lemma}\label{lemma-EqN1} Let $\Sigma$ be an Alexandrov-Nirenberg surface
in $\mathbb R^3$ of class $C^5$ and $\sigma$ be a connected component in $\partial\Sigma$.
Then, in the geodesic
coordinates as in \eqref{geod1} and \eqref{geod2}, with
$(\partial_1, \partial_2)=(\partial_s, \partial_t)$,
\begin{align}\label{eq5/N}
a^{ij}\partial_{ij}\left(\frac
1N\right)+\frac {A_{i}}N\partial_i\left(\frac
1N\right)&=\frac{A_0}{N^2},\\
\label{eq6/N}
\partial_i\left(a^{ij}\partial_{j}\left(\frac
1N\right)\right)+\frac {\widetilde A_{i}}{N}\partial_i\left(\frac
1N\right)&=\frac{A_0}{N^2},\\
\label{eq7/N}\partial_i\left(Na^{ij}\partial_{j}\left(\frac 1
{N^2}\right)\right)+ \widetilde A_{i}\partial_i\left(\frac 1
{N^2}\right)&=\frac{2A_0}{N^2},
\end{align}
where
$$a^{11}=N, \quad a^{12}=-M, \quad
a^{22}=L,$$
and $A_1, A_2, \widetilde A_1, \widetilde A_2$ and $A_0$ are polynomials of $M, N, B, B^{-1}, K$ and
derivatives of $B, K$, with $A_1, A_2, \widetilde A_1, \widetilde A_2$
involving derivatives of $B,K$ of order 1 and
$A_0$ involving derivatives  up to order 2. Moreover,
\begin{align} \label{eq11/N}
A_{2}&=2K_tB^2+BB_tN^2-\frac{4B_t}{B}M^2+4BB_tK,\\
\label{eq12/N}
\widetilde
A_{2}&=2K_tB^2+\frac{B_s}{B}MN-\frac{5B_t}{B}M^2
+3BB_tK.\end{align}
\end{lemma}

\begin{proof} By (\ref{GC3}), we
have
$$L=\frac {M^2+KB^2}N.$$
This makes sense in a neighborhood of $\sigma$ in $\Sigma$ as $N>0$
there by Lemma \ref{Lemma1.1}. Then (\ref{GC1}) and (\ref{GC2})
are reduced to
\begin{align}\label{sysNM1}\begin{split} -LN_t+2MN_s-NM_s+Q&=0,\\
M_t-N_s+\frac{B_t}{B}M&=0,\end{split}\end{align}
where
\begin{equation*}
Q=-\frac{2B_t}{B}M^2-\frac{B_t}{B}\big(M^2+KB^2\big)+\frac{B_s}{B}MN-BB_tN^2+(B^2K)_t.
\end{equation*}
We differentiate the first and second
equations in $t$ and in $s$ respectively and add the first resulted
equation and the $N$-multiple
of the second equation. Then,
\begin{align*}-LN_{tt}+2MN_{st}-NN_{ss}+2M_tN_s-L_tN_t-N_tM_s& \\
+\frac{B_t}{B}NM_s+\left(\frac{B_t}{B}\right)_sMN+Q_t&=0.\end{align*}
Note, by \eqref{sysNM1}
\begin{align*}
M_s&=\frac{2M}{N}N_s-\frac{L}{N}N_t+\frac{Q}{N},\\
M_t&=N_s-\frac{B_t}{B}M.\end{align*}
Also, by (\ref{GC1}),
$$L_t=M_s+\frac{B_t}{B}L-\frac{B_s}{B}M+BB_tN.$$
By a simple substitution of $L_t, M_s$ and $M_t$, we have
$$-LN_{tt}+2MN_{st}-NN_{ss}
+\frac{2L}{N}N_t^2-\frac{4M}{N}N_sN_t+2N_s^2-I=0,$$
where
\begin{align*}I&=\frac{2B_t}{B}MN_s+\frac{2Q}{N}N_t
+\left(\frac{B_t}{B}\frac {M^2+KB^2}N-\frac{B_s}{B}M+BB_tN\right)N_t\\
&\qquad -\frac{B_t}{B}NM_s-\left(\frac{B_t}{B}\right)_sMN-Q_t.\end{align*}
By substituting $M_s$ and $M_t$ by $N_s$ and $N_t$,
we note that $I$ is linear in $N_s$ and $N_t$ and hence can be put in the form
$$I=\frac {A_1}NN_s+\frac  {A_2}NN_t+A_0,$$
where $A_1, A_2$ and $A_0$ are polynomials of $M, N, B, B^{-1}, K$ and
derivatives of $B, K$, with $A_1, A_2$ involving derivatives of $B,K$ of order 1 and
$A_0$ involving derivatives  up to order 2.
Then,
\begin{align}\label{eq01/N}\begin{split}
-LN_{tt}+2MN_{st}-NN_{ss}+\frac {2L}NN_t^2-\frac {4M}{N}N_sN_t+2N_s^2& \\
- \frac{A_1}{N}N_s- \frac{A_2}{N}N_t- A_0&=0.\end{split}\end{align}
In the following, we need an explicit expression of $A_2$. Indeed, by the
expression of $Q$ and a straightforward calculation, we obtain
\eqref{eq11/N}.

By dividing \eqref{eq01/N} by $1/N^2$, we obtain
\begin{align}\label{eq1/N}\begin{split}&L\partial_{tt}\left(\frac
1N\right)-2M\partial_{ts}\left (\frac
1N\right)+N\partial_{ss}\left(\frac 1N\right)\\
&\qquad+\frac
{A_{2}}N\partial_t\left(\frac 1N\right)+\frac
{A_{1}}N\partial_s\left(\frac
1N\right)-\frac{A_0}{N^2}=0.\end{split}\end{align}
This is \eqref{eq5/N}. We can also express (\ref{eq1/N}) in divergence form
\begin{align}\label{eq3/N}\begin{split}\partial_t\left[L\partial_{t}
\left(\frac 1N\right)-M\partial_{s}\left(\frac
1 N\right)\right]+\partial_s\left[N\partial_{s}\left(\frac
1N\right)-M\partial_{t}\left(\frac
1 N\right)\right]& \\
+\frac {\widetilde A_{2}}{N}\partial_t\left(\frac 1N\right)+\frac
{\widetilde A_{1}}{N}\partial_s\left(\frac 1N\right) -
\frac{A_0}{N^2}&=0,\end{split}\end{align} where
$$\widetilde A_1=A_{1}+N(M_t-N_s),\quad \widetilde A_2=A_{2}-N(L_t-M_s).$$
This yields \eqref{eq6/N} and \eqref{eq12/N} by (\ref{GC1}) and (\ref{GC2}).

Last, \eqref{eq7/N} follows from \eqref{eq6/N}.
\end{proof}

\begin{remark}\label{rmk-Coeff-N} The explicit
expressions of $A_2$ and $\widetilde A_2$ play an important role in the estimate of
$1/N$.
We note that, for  any
$t\in [0,\delta]$,
\begin{equation}\label{eq4/N}A_{2} \ge C_1-C_2M^2,\end{equation}
where $C_1, C_2$ and
$\delta$ are positive constants under control. This follows easily
from  $K=0$, $K_t>0$ at $t=0$ and $B_t>0$ in
the region considered.
\end{remark}


Now, we consider a function $h=h(s,t)$ and derive an equation for
$\frac{1}{N^2}-h$. By \eqref{eq7/N}, we have
$$\partial_i\left(Na^{ij}\partial_{j}\left(\frac 1
{N^2}-h\right)\right)+ \widetilde A_{i}\partial_i\left(\frac 1
{N^2}-h\right)-\frac{2A_0}{N^2}+I=0,
$$
where
\begin{align*}I&= \partial_i(Na^{ij}h_j)+\widetilde A_ih_i\\
&=N\partial_i( a^{ij}h_j)+\widetilde
A_ih_i+a^{ij}h_j\partial_iN.
\end{align*}
For the last term, we write it as
$$a^{ij}h_j\partial_iN=-\frac12N^3a^{ij}h_j\partial_i\left(\frac1{N^2}\right)
=-\frac12N^3a^{ij}h_j\partial_i\left(\frac1{N^2}-h\right)
-\frac12N^3a^{ij}h_jh_i.$$
Therefore, by a simple substitution, we obtain
\begin{equation}\label{eq8/N}\partial_i\left(Na^{ij}\partial_{j}\left(\frac 1
{N^2}-h\right)\right)+ \overline{A}_{i}\partial_i\left(\frac 1
{N^2}-h\right)-\overline{A}_0=0,
\end{equation}
where
\begin{equation*}
\overline{A}_i=\widetilde A_i-\frac12N^3a^{ij}h_j,
\end{equation*} and
\begin{equation*}
\overline{A}_0= \frac{2A_0}{N^2}-Na^{ij}h_{ij}-\big(\widetilde A_i+N\partial_ja^{ij}\big)h_i
+\frac12
N^3a^{ij}h_ih_j.\end{equation*}
In the expression of $\overline{A}_0$, the derivatives of $a^{ij}$ have the form
$$\partial_1a^{11}+\partial_2a^{21}=\partial_sN-\partial_tM,\quad
\partial_1a^{12}+\partial_2a^{22}=\partial_tL-\partial_sM,$$
which can be substituted by the Codazzi equations \eqref{GC1}-\eqref{GC2}.
In the special case $h=h(s)$, we have
\begin{equation}\label{eq9/N}
\overline{A}_1=\widetilde A_1-\frac12N^4h_s,\quad
\overline{A}_2=\widetilde A_2+\frac12MN^3h_s,\end{equation}
and
\begin{equation}\label{eq10/Na}
\overline{A}_0=
\frac{2A_0}{N^2}-N^2h_{ss}-\widetilde A_1h_s-\frac{B_t}{B}MNh_s+\frac12N^4h_s^2.\end{equation}

Next, we derive an equation for $M$.

\begin{lemma}\label{lemma-EqM1} Let $\Sigma$ be an Alexandrov-Nirenberg surface
in $\mathbb R^3$ of class $C^5$ and $\sigma$ be a connected component in $\partial\Sigma$.
Then, in the geodesic
coordinates as in \eqref{geod1} and \eqref{geod2},
with $(\partial_1, \partial_2)=(\partial_s, \partial_t)$,
\begin{align}\label{eq1M}
a^{ij}\partial_{ij}M-\frac{2M}{NL}a^{ij}\partial_iM\partial_jM+\widehat A_{i}\partial_iM
+\widehat A_0&=0,\\
\label{eq2M}
\partial_i(a^{ij}\partial_{j}M)-\frac{2M}{NL}a^{ij}\partial_iM\partial_jM+\widecheck{A}_{i}\partial_iM
+\widehat A_0&=0,
\end{align}
where
$a^{11}=N$, $ a^{12}=-M$, $a^{22}=L$, and
\begin{align}\label{eq3M}\begin{split}
\widehat A_0&=\frac{1}{NL}(KB^2)_s\widehat A_{01}+\widehat
A_{02}\\
 \widehat A_1&=\frac{2M}{NL}(B^2K_t-BB_tN^2)+\frac{4K}{L}\left(BB_t\frac
MN-BB_s\right)
-\frac{B^2K_s}{L}+\widehat A_{11},\\
\widehat A_2&=\frac{2M}{NL}(KB^2)_s+\widehat A_{21},
\end{split}\end{align}
and
$$\widecheck{A}_1=\widehat A_1-\frac{B_t}{B}M,\quad
\widecheck{A}_2=\widehat A_2-\frac{B_t}{B}L+\frac{B_s}{B}M-BB_tN,$$
and $\widehat A_{01},
\widehat A_{02}, \widehat A_{11} $ and $\widehat A_{21}$
are polynomials of $M,N,N^{-1}$, $B, B^{-1}$, $K$, and  derivatives of $B, K$ up to order 2.
\end{lemma}

\begin{proof} By a simple arrangement, we write (\ref{sysNM1}) as
\begin{align}\label{sysNM2}\begin{split}-LN_t+2MM_t-NM_s+\widehat{Q}&=0,\\
M_t-N_s+\frac{B_t}{B}M&=0,\end{split}\end{align} where
$$
\widehat Q=(KB^2)_t-\frac {B_t}B(M^2+KB^2)+\frac {B_s}BMN-BB_tN^2.
$$
By eliminating the derivatives of $N$ in \eqref{sysNM2}, we will get a differential equation of $M$.
Specifically, we differentiate the first and second equations in
$s$ and in $t$ respectively,
multiply the second and the first equations of (\ref{sysNM2})
by $L$ and $-1$, and then sum up. We then have
\begin{align*}LM_{tt}-2MM_{st}+NM_{ss}+L_sN_t+N_sM_s-2M_sM_t& \\
+\frac{B_t}{B}LM_t+\left(\frac{B_t}{B}\right)_tLM -\widehat
Q_s&=0.\end{align*}
By \eqref{sysNM2}, we have
\begin{align*}
N_t&=\frac{2M}{L}M_t-\frac{N}{L}M_s+\frac{\widehat Q}{L},\\
N_s&=M_t+\frac{B_t}{B}M.\end{align*}
Next,
$$L_s=\left(\frac{M^2+KB^2}{N}\right)_s
=\frac{2M}{N}M_s+\frac1N(KB^2)_s-\frac{M^2+KB^2}{N^2}N_s. $$
By substituting $N_s$ in the expression of $L_s$, we have
$$L_s=\frac{2M}{N}M_s-\frac{L}{N}M_t-\frac{B_t}{B}\frac{LM}{N}+\frac1N(KB^2)_s.$$
By a simple substitution of $L_s$, $N_s$ and $N_t$, we obtain
$$LM_{tt}-2MM_{st}+NM_{ss}
-\frac{2M}{N}M_t^2+\frac{4M^2}{NL}M_sM_t-\frac{2M}{L}M_s^2+\widehat I=0,$$
where
\begin{align*}
\widehat I&=\left(-\frac{B_t}{B}\frac{LM}{N}
+\frac{1}{N}(KB^2)_s\right)\left(\frac{2M}{L}M_t-\frac{N}{L}M_s\right)\\
&\qquad +\frac{\widehat
Q}{L}\left(\frac{2M}{N}M_s-\frac{L}{N}M_t\right)
+\frac{\widehat Q}{L}\left(-\frac{B_t}{B}\frac{LM}{N}+\frac1N(KB^2)_s\right)\\
&\qquad+\frac{B_t}{B}MM_s+\frac{B_t}{B}LM_t+\left(\frac{B_t}{B}\right)_tLM -\widehat
Q_s.\end{align*}
Note that $\widehat I$ is linear in $M_s$ and $M_t$ and hence can be put in the form
$$\widehat I=\widehat A_1M_s+\widehat A_2M_t+\widehat A_0,$$
for some functions $\widehat A_0$, $\widehat A_1$ and $\widehat A_2$.
Then,
\begin{align}\label{mequation1}\begin{split}
LM_{tt}-2MM_{st}+NM_{ss}-\frac {2M}NM_t^2+\frac {4M^2}{LN}M_sM_t-\frac {2M}LM_s^2& \\
+\widehat A_1M_s+\widehat A_2M_t+\widehat A_0&=0.\end{split}\end{align}
In calculating $\widehat I$, we need to collect terms involving $L^{-1}$.
By the explicit expression of $\widehat Q$ and a
straightforward calculation, we obtain
\eqref{eq3M}.
We can also write \eqref{mequation1} in the
divergence form
\begin{align}\label{mequation1p}\begin{split}
(LM_{t}-MM_{s})_t+(NM_{s}-MM_t)_s
-\frac {2M}{NL}(LM_t^2-2MM_sM_t+NM_s^2)& \\
+\left(\widehat A_1-\frac{B_t}{B}M\right)M_s
+\left(\widehat A_2-\frac{B_t}{B}L+\frac{B_s}{B}M-BB_tN\right)M_t
+\widehat A_0&=0.\end{split}\end{align}
Note that \eqref{mequation1} and \eqref{mequation1p} are \eqref{eq1M} and
\eqref{eq2M}, respectively.
\end{proof}

We now analyze $\widehat A_0$, $\widehat A_1$ and $\widehat A_2$.

\begin{remark}\label{rmk-Coeff-M} We may write
$$\widehat A_1=\widehat A_{12}\frac{M}{L}\big(N-N(s,0)\big)+
\widehat A_{13}.$$
Then,  $\widehat A_{12}$, $\widehat A_{13}$,
$\widehat A_0$ and $\widehat A_2$ are bounded by a constant depending only on
\begin{equation}\label{maincond1}
\sup_{0\le t\le  1,\ |\alpha|\le 2} \left\{M,N,\frac 1N,\frac{1}{K_t},
\frac 1{B},|\partial^{\alpha}B|, |\partial^\alpha K|\right\}.
\end{equation}\end{remark}

To see this, we note that, by $LN=KB^2+M^2\ge KB^2,$ we have
$$\frac{K}{L}\le \frac{N}{B^2}.$$
By $K=0,K_t> 0$ at $t=0$, we have
$$K_s\le CK\le CL.$$
Hence, $\widehat A_0$, $\widehat A_2$ are bounded and $\widehat
A_1$ is bounded by \eqref{maincond1} except the term
$$\frac{2M}{NL}(B^2K_t-BB_tN^2).$$
We write this as
$$-\frac{2BB_tM}{NL}\big(N^2-N^2(s,0)\big)
+\frac{2BM}{NL}\big(BK_t-B_tN^2(s,0)\big).$$ By Lemma
\ref{Lemma1.1}, we have $\big(BK_t-B_tN^2\big)(s,0)=0$ on $t=0$.
Also, since $K_t>0$ on $t=0$, we have $L\ge Ct$. Therefore, the
second term above is bounded by \eqref{maincond1}. We then have the
desired decomposition of
$\widehat A_1$.

\section{$L^\infty$-estimates near Boundary}\label{sec-BoundsMeanCurv}

Let $\Sigma$ be an Alexandrov-Nirenberg
surface in $\mathbb R^3$.
Starting from this section, we will estimate the second fundamental form near
boundary $\partial\Sigma$. We first prove an $L^\infty$-estimate by the maximum principle.

\begin{lemma}\label{Remark1.13} Let $\Sigma$ be an Alexandrov-Nirenberg
surface in $\mathbb R^3$ of class $C^5$ and $\sigma$ be a connected component of
$\partial\Sigma$. Then, in the geodesic
coordinates as in (\ref{geod1})-(\ref{geod2}),
$$0\le L\le C,\quad |M|\le C, \quad \frac1C\le N\le C\quad\text{for any }
t\in [0, \delta_0],$$ where $C$ and $\delta_0$ are positive
constants depending only on the quantities
\begin{equation}\label{maincond}
|g|_{C^4(\Sigma)},\quad \max_{\partial \Sigma}\frac{1}{|k_g|},\quad
\max_{\partial \Sigma}\frac{1}{|\nabla K|}.\end{equation}
\end{lemma}

\begin{proof}
First, $L\ge 0$. Lemma
\ref{Lemma1.11} and \eqref{mean} imply $N\le C$ and $L\le CB^2$.
Then, (\ref{GC3}) yields $|M|\le\sqrt {LN}\le CB.$

Next, in the
geodesic coordinates as in (\ref{geod1}) and (\ref{geod2}), the
normal curvature in the direction of $\partial_t$ at $p$
equals $N$. We have by \eqref{eq-PrincipalLowerBound}, for $t=\delta_0>0$,
\begin{equation}\label{H3}
N=II(\partial_t,\partial_t)\ge \min\{k_1,k_2\}\ge \frac 1{C_\delta}.
\end{equation}
We now claim there exists a $\delta_0\in (0,1]$ such that, for any $t\in (0,\delta_0)$,
\begin{equation}\label{H3q}N\ge \frac1C,\end{equation} where $C$ is a positive constant
depending only on those quantities in \eqref{maincond}.
To prove this, we set
$$\Omega_{\delta_0}=\{(s,t):\, s\in[0,2\pi],\ t\in(0,\delta_0)\}.$$
We write \eqref{eq5/N} as
\begin{align*}L\partial_{tt}\left(\frac
1N\right)-2M\partial_{ts}\left (\frac
1N\right)+N\partial_{ss}\left(\frac 1N\right)&\\
+\frac {A_{2}}N\partial_t\left(\frac 1N\right)+\frac
{A_{1}}N\partial_s\left(\frac
1N\right)&=\frac{A_0}{N^2}.\end{align*} Since
$L, M, N$ are all bounded, then, $A_1$,
$A_2$ and  $A_0$  are bounded. Moreover,  $A_2$ satisfies \eqref{eq4/N}.
Set
$$\phi =\frac hN,$$
for a function $h=h(t)$ to be fixed. Then,
\begin{align*}\mathcal L\phi &\equiv L\phi_{tt} -2M\phi_{ts}
+N\phi_{s s} +\left(\frac
{A_{1}}N+\frac{2Mh_t}{h}\right)\phi_s+\left(\frac {A_{2}}N-\frac
{2Lh_t}h\right)\phi_t\\
&=\frac L{Nh}(h_{t t}h-2h^2_t)+A_{2}h_t\frac
1{N^2}+h\frac{A_0}{N^2}.\end{align*}
Set
$$h=\frac{1}{\sqrt {1-\lambda t}}.$$ Then, for $\lambda t<1,$
$$\frac{h_t}{h}=\frac{\lambda}{2(1-\lambda t)},\quad \frac{h_{tt}}{h}
-\frac{2h^2_t}{h^2}=
\frac{\lambda^2}{4(1-\lambda t)^2}.$$
Both expressions are positive. Hence,
by \eqref{eq4/N}, we have, for any $t\in [0,\delta]$, where $\delta$ is
introduced for  \eqref{eq4/N},
\begin{align*}
\frac1h\mathcal L\phi &\ge \frac {M^2}{N^2h^2}(h_{t
t}h-2h^2_t)+\frac
1{N^2}(C_1-C_2M^2)\frac{h_t}{h}+\frac{A_0}{N^2}\\
&\ge\frac {M^2}{N^2}\left(\frac {\lambda^2}{4(1-\lambda
t)^{2}}-\frac { \lambda C_2}{2(1-\lambda t)}\right)+\frac
1{N^2}\left(\frac {\lambda C_1}{2(1-\lambda t)}-|A_0|\right).
\end{align*}
Take
$$\lambda =2C_2+\frac{2}{C_1}\big(\max|A_0|+1\big),$$
and then $\delta_0=\min \{\delta,(2\lambda )^{-1}\}.$ We have, for any $t\in (0,\delta_0)$,
$$\frac 1h\mathcal L\phi\ge\frac {1}{N^2}>0.$$
Assume that $\phi$ attains its maximum in $\bar \Omega_{\delta_0}$ at
some point $p=(s_p, t_p)\in\bar{\Omega}_{\delta_0}.$ The maximum principle implies that
$p\in\partial \Omega_{\delta_0}$.
If $t_p=\delta_0$, then \eqref{H3} implies $N\ge 1/C_{\delta_0}$ and hence
$\phi\le \sqrt 2C_{\delta_0}$ at $p$. If $t_p=0$, Lemma \ref{Lemma1.1} yields a similar estimate.
Hence, $\phi\le C_*$ in $\Omega_{\delta_0}$ and then $N\ge 1/(\sqrt{2}C_*)$.
This finishes the proof of \eqref{H3q}.
\end{proof}

For simplicity, we will write $\delta_0=1$ in the following.
Next, we derive an estimate of $M$.

\begin{lemma}\label{Lemma1.17}
Let $\Sigma$ be an Alexandrov-Nirenberg
surface in $\mathbb R^3$ of class $C^5$ and $\sigma$ be a connected component of
$\partial\Sigma$.
Then, in the geodesic coordinates as in (\ref{geod1}) and
(\ref{geod2}),
$$|M|\le C\sqrt t \quad\text{for any }t \in [0,1],$$
where $C$ is a positive constant depending only on the
quantities in (\ref{maincond}).
\end{lemma}

\begin{proof} Note that $M$ satisfies \eqref{eq1M} or (\ref{mequation1}).
Set $m=M^2$. Multiplying both sides of (\ref{mequation1}) by $2M$ yields
\begin{align*}
\mathcal R(m)\equiv Lm_{tt}-2Mm_{st}+Nm_{ss}+\widehat A_1m_s+\widehat A_2m_t&\\
-\frac 1{N}m_t^2+\frac {2M}{LN}m_sm_t-\frac {1}{L}m_s^2&\ge
-2\widehat A_0M.\end{align*}
It is easy to see
$$\mathcal R(\lambda t)=\lambda\left(\widehat A_2-\frac {\lambda}{N}\right)
\le -\frac {\lambda^2}{C_1},$$ for some positive constant $C_1$
and sufficiently large constant $\lambda$ under control since
$\widehat A_2$ is bounded by
Remark \ref{rmk-Coeff-M}. Set
$$w=m-\lambda t.$$
By taking the difference $\mathcal R(m)-\mathcal R(\lambda t)$, we have
for any $(s,t)\in [0,2\pi]\times (0,1)$,
\begin{align*}
Lw_{tt}-2Mw_{st}+Nw_{ss}+\left(\widehat A_1+\frac {2Mm_t}{LN}-\frac {m_s}{L}\right)w_s&\\
+\left(\widehat A_2-\frac 1{N}(m_t+\lambda)\right)w_t& \ge \left(\frac
{\lambda^2}{C_1}-C_2\right)>0,\end{align*}
for sufficiently large $\lambda$ under control, since $\widehat A_0$ is bounded by
Remark \ref{rmk-Coeff-M}. Note that
$w=0$ as $t=0$ and $w=m-\lambda\le 0$ as $t=1$ by choosing $\lambda$ large,
since $M$ is bounded. By the maximum principle, we conclude $w\le 0$ and hence,
for any $ (s,t)\in[0,2\pi]\times [0,1]$,
\begin{equation*}\label{mequation4}
|M|^2\le \lambda t.
\end{equation*}
This yields the desired result.
\end{proof}

Our next step is to estimate $|N(s,t)-N(s,0)|$. However,
the barrier argument does not seem to work for this purpose
from the equation (\ref{eq5/N}) since it is characteristically
degenerate along boundary $t=0$. This is the major obstacle we encounter.
We have to employ  different methods in the next two sections.

\section{H\"{o}lder Estimates near Boundary}\label{sec-BoundaryHolder}

In this section, we derive the boundary H\"{o}lder estimates of $N$
in the geodesic coordinates.
The main
technique is the de Giorgi iteration.

We first prove some basic results concerning  weighted
Sobolev spaces. For a domain $G\subset
\mathbb R^2_+=\{(s,t)\in \mathbb R^2:\, t>0\}$, denote by
$\widetilde W^{1,2}(G)$
the completion of
$C^1(\bar G)$ under the  norm
$$\left(\int_G (tu_t^2+u_s^2+u^2)dtds\right)^{\frac 12}.$$
For any $p_0=(s_0,0)$ and any $R>0$, set
$$G_R(p_0)=\{(s,t):\, |s-s_0|<\sqrt{R},\ 0<t<R\}.$$
If no confusion occurs, we simply
write $G_R$.

\begin{lemma}\label{lemma-Sobolev}
$\operatorname{(1)}$ For any $u\in \widetilde W^{1,2}(G)$ with $u=0$ on
$\partial G\cap \mathbb R^2_+$,
\begin{equation}\label{weight1}
\left(\int_{G}u^6dsdt\right)^\frac 13\le C
\int_{G}(tu_t^2+u_s^2)dsdt,
\end{equation}
where $C$  is a universal positive constant,  independent of $G$.

$\operatorname{(2)}$ For any $\epsilon >0$ and
any $u\in C^1(\bar G_1)$ with $|\{(s,t)\in G_1:\,
u(s,t)=0\}|\ge \epsilon$,
\begin{equation}\label{weight2}
\int_{G_1}u^2dsdt\le C_\epsilon \int_{G_1}(tu_t^2+u_s^2)dsdt,
\end{equation}
where $C_\epsilon$ is a positive constant depending only on $\epsilon$.
\end{lemma}

The proof is based on the
raising dimension argument.

\begin{proof} Let $G\subset \mathbb R^2_+$ and $u\in \widetilde W^{1,2}(G)$.
Define a transform $T: G\to T(G)$ by
$$ T(s,t)=(s,\tau)\equiv (s,2\sqrt t). $$
Lift
$T(G)$ in $\mathbb R^3$ by defining
$$\widetilde{T(G)}
=\{(s,\tau,\lambda)\in \mathbb R^3:\, (s,\tau)\in T(G), \, 0<\lambda<\tau\}.$$
Then
\begin{equation}\label{weight3} \int_G |u|^pdsdt=\frac 12\int_{T(G)}|u\circ
T^{-1}|^p\tau dsd\tau=\frac 12 ||u\circ
T^{-1}||^p_{L^p(\widetilde{T((G))})},\end{equation} and
\begin{equation}\label{weight4}
\int_G(tu_t^2+u_s^2)dsdt=\frac 12\int_{T(G)}(u_\tau^2+u_s^2)\tau
dsd\tau=\frac 12\|\widetilde \nabla( u\circ
T^{-1})\|^2_{L^2(\widetilde{T(G)})}, \end{equation} where $\widetilde
\nabla=(\partial_s,\partial_\tau,\partial_\lambda)$ is the
gradient in $\mathbb R^3$.

Now let us consider the first part of the
present lemma. It suffices to prove (\ref{weight1})  for all
$u\in C^1(\bar G)$. Let $u\in C^1(\bar G)$ with $u=0$ on $\partial
G\cap \mathbb R^2_+$. Set
$$\tilde u(s,t)=\begin{cases}u(s,t)&\text{for }(s,t)\in G,\\ 0&\text{for }(s,t)
\in\mathbb R^2_+\setminus G.\end{cases}$$
Then, define $v(s,\tau)=\tilde u(s,t)$ and
$$w(s,\tau,\lambda)=\begin{cases}
v(s,\tau)&\text {for }(s,\tau,\lambda)\in\widetilde{T(G)}, \\
v(s,\lambda)&\text {for }\lambda>\tau>0.
\end{cases}$$
Then, by the Sobolev extension, we can extend $w$ to $\mathbb R^3$
by extensions first with respect to the plane $\tau=0$ and
then to the plane $\lambda=0$. By the Sobolev embedding, we have
$w\in H^1(\mathbb R^3) \subset L^6(\mathbb R^3)$  and
$$\left(\int_{\mathbb R^3}w^6dsd\tau d\lambda\right)^{\frac13}\le
C\int_{\mathbb R^3}|\widetilde{\nabla}w|^2dsd\tau d\lambda.$$
 Therefore by
(\ref{weight3}) and (\ref{weight4}), we obtain
\begin{align*}
\left(\int_Gu^6dsdt\right)^{\frac 13}
&\le \left(\frac1{2}\int_{\mathbb R^3}w^6dsd\tau d\lambda\right)^{\frac 13}
\le C\int_{\mathbb R^3}|\widetilde{\nabla}w|^2dsd\tau d\lambda\\
&\le C'\int_{\widetilde{T(G)}}|\widetilde \nabla w|^2dsd\tau d\lambda
=2C'\int_G(tu_t^2+u_s^2)dsdt,
\end{align*}
where $C$ and $C'$ are universal positive constants, independent of $u$.

Next, we consider the second part of the present lemma. Suppose that
$u\in C^1(\bar G_1)$ with $|\{(s,t)\in G_1:\, u(s,t)=0\}|\ge \epsilon>0$. Then,
$$|\{(s,\tau):\,v(s,\tau)=0\}|\ge C\epsilon,$$
and
$$|\{(s,\tau,\lambda)\in \widetilde{T(G_1)}:\, w(s,\tau,\lambda)=0\}|\ge C\epsilon,$$
for some universal
constant $C$.
We now
extend $\widetilde{T(G_1)}$ by reflecting
$\widetilde{T(G_1)}$ with respect to $\lambda=\tau$ to get a
domain $\widehat{T(G_1)}$.
By the well-known Poincar\'{e} inequality,  we get
$$\int_{\widehat{T(G_1)}}w^2dsd\tau d\lambda\le C_\epsilon
\int_{\widehat{T(G_1)}}|\nabla w|^2dsd\tau d\lambda,$$
where $C_\epsilon$ is a positive constant depending only on $\epsilon$. Then
\begin{align*}
\int_{G_1}u^2dsdt &=\int_{\widehat{T(G_1)}}w^2dsd\tau d\lambda
\le C_\epsilon \int_{\widehat{T(G_1)}}|\widetilde
\nabla w|^2dsd\tau d\lambda\\
&\le C_\epsilon\int_{G_1}(tu_t^2+u_s^2)dsdt.
\end{align*}
This completes the proof of the present lemma.
\end{proof}

Next, we discuss the boundary regularity of $N$. We will first formulate several results
for a general class of elliptic equations which are degenerate on boundary.
Consider
\begin{equation}\label{eq-Deg0}
\partial_i(a_{ij}\partial_ju)+b_i\partial_iu=f\quad\text{in }\mathbb R^2_+,\end{equation}
where we write $(\partial_1,\partial_2)=(\partial_s,\partial_t).$ We first assume,
for some positive constant $C_*>0$,
\begin{equation}\label{eq-Deg1}
C_*^{-1}(\xi_1^2+t\xi_2^2)\le a_{ij}\xi_i\xi_j\le C_*(\xi_1^2+t\xi_2^2)
\quad\text{for any }\xi\in\mathbb R^2.\end{equation}
We then have $C_*^{-1}\le a_{11}\le C_*$ by taking $\xi_2=0$ and $C_*^{-1}t\le a_{22}\le C_*t$ by taking
$\xi_1=0$. Then $a_{12}^2\le a_{11}a_{22}\le C_*^2t$. In particular, we have
$a_{2j}=0$ on $t=0$. Concerning $b_1$ and $b_2$,
we assume
$b_2=b_{21}+b_{22}$ such that
\begin{equation}\label{eq-Deg2}
b_{21}\ge 0\quad\text{on }t=0,\end{equation}
and
\begin{equation}\label{eq-Deg3}
|b_1|+ |b_{21}|+|\partial_tb_{21}|\le C_*,\quad |b_{22}|\le C_*\sqrt{t}.\end{equation}

We first derive an energy estimate for \eqref{eq-Deg0}.

\begin{lemma}\label{garding1} Let \eqref{eq-Deg1}, \eqref{eq-Deg2} and
\eqref{eq-Deg3} be assumed and
$u\in \widetilde{W}^{1,2}(\mathbb R^2_+)\cap C^1(\bar{\mathbb R}^2_+)$ satisfy
\eqref{eq-Deg0}. Then, for any
$\varphi\in C_0^\infty(\bar{\mathbb R}^2_+)$,
\begin{equation} \label{401}\int
\varphi^2\left(tu_t^2+u_s^2\right)\leq
C\int
(\varphi^2+t\varphi_t^2+\varphi_s^2+\varphi|\varphi_t|)u^2+\int
\varphi^2f^2,
\end{equation}
where $C$ is a positive constant depending only on $C_*$ in
\eqref{eq-Deg1} and \eqref{eq-Deg3}.\end{lemma} Here we emphasize
that $\varphi$ is not assumed to be zero on $\{t=0\}$.

\begin{proof} We multiply \eqref{eq-Deg0} by
$-\varphi^2 u$
and integrate by parts. Let $G\subset\mathbb R^2_+$ be a domain such that
$\varphi=0$ in $\mathbb R^2_+\setminus G$. Then,
$$\int_G \varphi^2a_{ij}u_iu_j
=\int_{\partial G}\varphi^2ua_{ij}u_j\nu_i
-2\int_G\varphi ua_{ij} \varphi_iu_j
+\int_G\varphi^2ub_iu_i-\int_G\varphi^2uf.$$
For the boundary integral, we first note $\varphi=0$ on $\partial G\cap \mathbb R^2_+$.
Next, on $\partial G\cap \{t=0\}$, $\nu_1=0$ and $a_{2j}=0$. Hence, boundary integrals
are absent from the expression above.
Next, the Cauchy inequality implies, for $\varepsilon>0$ to be determined,
$$2\varphi ua_{ij}\varphi_i u_j\le \varepsilon \varphi^2a_{ij}u_iu_j+\frac{1}{\varepsilon}
a_{ij}\varphi_i\varphi_j u^2.$$
By \eqref{eq-Deg1}, we have
$$a_{ij}u_iu_j\ge C_*^{-1}(u_s^2+tu_t^2),$$
and
$$a_{ij}\varphi_i\varphi_j\le C_*(\varphi_s^2+t\varphi_t^2).$$
Therefore,
$$\frac{1}{C_*}(1-\varepsilon)\int \varphi^2(u_s^2+tu_t^2)
\le \frac{C_*}{\varepsilon}\int(\varphi_s^2+t\varphi_t^2)u^2
+\int\varphi^2ub_iu_i-\int\varphi^2uf.$$
Next, for the $b_1$-term, we have, by $|b_1|\le C_*$ in \eqref{eq-Deg3},
$$\left|\int \varphi^2ub_1u_s\right|\le \frac{\varepsilon}{C_*}\int\varphi^2u_s^2
+\frac{C_*^3}{\varepsilon}\int\varphi^2u^2.$$
For the $b_2$-term, we write $b_2=b_{21}+b_{22}$ and have, by
$|b_{22}|\le C_*\sqrt t$ in \eqref{eq-Deg3},
$$\left|\int \varphi^2ub_{22}u_t\right|\le \frac{\varepsilon}{C_*}\int\varphi^2tu_t^2
+\frac{C_*^3}{\varepsilon}\int\varphi^2u^2.$$
On the other hand,
$$\int_G\varphi^2b_{21}uu_t=\frac12\int_G\varphi^2b_{21}(u^2)_t
=\frac12\int_{\partial G}\varphi^2b_{21}u^2\nu_t-\frac{1}{2}\int_G(\varphi^2b_{21})_tu^2.$$
On $t=0$, $\nu_2=-1$ and $b_{21}\ge 0$, and hence
$$\int_{\partial G}\varphi^2b_{21}u^2\nu_t\le 0.$$
Therefore, by $|b_{21}|+|\partial_tb_{21}|\le C_*$ in \eqref{eq-Deg3},
$$\int \varphi^2b_{21}uu_t\le -\frac{1}{2}\int(\varphi^2\partial_tb_{21}+2\varphi\varphi_tb_{21})u^2
\le C_*\int(\varphi^2+\varphi|\varphi_t|)u^2.$$
By a simple substitution and taking $\varepsilon=1/4$,   we obtain
$$\int \varphi^2(u_s^2+tu_t^2)
\le C_1\int(\varphi_s^2+t\varphi_t^2+\varphi^2+\varphi|\varphi_t|)u^2
+C_2\int\varphi^2|uf|.$$
Another application of the Cauchy inequality implies the desired result. \end{proof}

In the following, we study estimates of H\"{o}lder norms of solutions near boundary.
Our main tool is an iteration due to de Giorgi. We will follow \cite{Han-Lin2011}
closely.

First, we prove a local $L^\infty$-estimate for subsolutions.

\begin{lemma}\label{moser2}
Let \eqref{eq-Deg1}, \eqref{eq-Deg2} and
\eqref{eq-Deg3} be assumed and $f\in L^q(G_R)$, for some $R\in (0,1]$
and $q>3/2$. Suppose
$u\in C^1(\bar G_R)$ satisfies
$$\partial_i(a_{ij}u_j)+b_iu_i\ge f\quad\text{weakly in }G_R.$$
Then, for any $\theta\in (0,1)$,
\begin{equation}\label{413}
\sup_{G_{\theta R}}u^+\le C\left\{\left(\frac
1{|G_R|}\int_{G_R}u^2\right)^{\frac 12}
+R\left(\frac
1{|G_R|}\int_{G_R}|f|^q\right)^{\frac 1q}\right\},\end{equation}
where $C$ is a positive constant depending only on
$q$, $\theta$ and $C_*$. \end{lemma}

\begin{proof} For simplicity, we assume $R=1$.
Let $\varphi $ be a smooth cutoff function with support in
$G_R\cup\{(s,0):\, |s|<1\})$ and
$0\le\varphi\le 1$, and set
$\bar u=(u-k)^+$ for some $k\ge 0$. Multiply
the differential inequality by $-\varphi^2\bar u$
and integrate in $G_1$. Proceeding as in the proof of Lemma \ref{garding1}, we have
$$\int\varphi^2(\bar u_s^2+t\bar u_t^2)
\le C\int\big(\varphi^2+\varphi_s^2+t\varphi_t^2+\varphi|\varphi_t|\big)\bar u^{2}+\int\varphi^2\bar uf,$$
and then
$$\int\big((\partial_s(\varphi\bar u))^2
+t(\partial_t(\varphi\bar u))^2\big)
\le C\int\big(\varphi^2+\varphi_s^2+t\varphi_t^2
+\varphi|\varphi_t|\big)\bar u^{2}+\int\varphi^2\bar uf.$$
Lemma
\ref{weight1}(1) implies
$$\left(\int\varphi^6\bar u^{6}\right)^{\frac13}
\le C\int\big(\varphi^2+\varphi_s^2+t\varphi_t^2+\varphi|\varphi_t|\big)\bar u^{2}
+\int\varphi^2\bar uf.$$
By the H\"{o}lder inequality, we have
\begin{align*}
\int\varphi^2\bar u f&\le \left(\int(\varphi\bar u)^6\right)^{\frac16}
\left(\int(\varphi f)^q\right)^{\frac1q}
|\{\varphi \bar u\neq 0\}|^{1-\frac16-\frac1q}\\
&\le \frac12\left(\int\varphi^6\bar u^{6}\right)^{\frac13}
+\frac12\|f\|_{L^q}^2|\{\varphi \bar u\neq 0\}|^{\frac53-\frac2q}, \end{align*}
and hence
$$\left(\int\varphi^6\bar u^{6}\right)^{\frac13}
\le C\int\big(\varphi^2+\varphi_s^2+t\varphi_t^2+\varphi|\varphi_t|\big)\bar u^{2}
+\|f\|_{L^q}^2|\{\varphi \bar u\neq 0\}|^{\frac53-\frac2q}.$$
By the H\"{o}lder inequality again, we have
$$\int(\varphi \bar u)^2\le \left(\int\varphi^6\bar u^{6}\right)^{\frac13}
|\{\varphi \bar u\neq 0\}|^{\frac23},$$
and hence
$$\int\varphi^2\bar u^2\le C\int\big(\varphi^2+\varphi_s^2+t\varphi_t^2+\varphi|\varphi_t|\big)\bar u^{2}
|\{\varphi \bar u\neq 0\}|^{\frac23}\\
+\|f\|_{L^q}^2|\{\varphi \bar u\neq 0\}|^{\frac73-\frac2q}.$$
In the following, we take
$$\varepsilon=\min\left\{\frac23, \frac43-\frac2q\right\}.$$
Then,
$$\int\varphi^2\bar u^2\le C\int\big(\varphi^2+\varphi_s^2+t\varphi_t^2+\varphi|\varphi_t|\big)\bar u^{2}
|\{\varphi \bar u\neq 0\}|^{\varepsilon}\\
+\|f\|_{L^q}^2|\{\varphi \bar u\neq 0\}|^{1+\varepsilon}.$$
Set, for any $r\in (0,1]$ and $k\ge 0$,
$$A(k,r)=\{(s,t)\in G_r:\, u(s,t)\ge k\}.$$
For any $0<r<R< 1$, we take a cutoff function $\varphi$ such that
$\varphi=1$ in $G_r$ and $\varphi=0$ in $G_1\setminus G_R$.
Then,
$$\varphi^2+\varphi_s^2+t\varphi_t^2+\varphi|\varphi_t|\le \frac{C}{(R-r)^2},$$
and hence
$$\int_{A(k,r)}(u-k)^2\le
C\left\{\frac{1}{(R-r)^2}\int_{A(k,R)}(u-k)^2|A(k,R)|^\varepsilon
+\|f\|_{L^q}^2|A(k,R)|^{1+\varepsilon}\right\}.$$
For any $h>k\ge 0$, we have
$$\int_{A(h,R)}(u-h)^2\le \int_{A(k,R)}(u-k)^2,$$
and
$$|A(h,R)|=|G_R\cap \{u-k>h-k\}|\le \frac{1}{(h-k)^2}\int_{A(k,R)}(u-k)^2.$$
Hence,
\begin{align*}
\int_{A(h,r)}(u-h)^2&\le
C\left\{\frac{1}{(R-r)^2}\int_{A(h,R)}(u-h)^2
+\|f\|_{L^q}^2|A(h,R)|\right\}|A(h,R)|^\varepsilon\\
&\le C\left\{\frac{1}{(R-r)^2}
+\frac{1}{(h-k)^2}\|f\|_{L^q}^2\right\}
\frac{1}{(h-k)^{2\varepsilon}}
\left(\int_{A(k,R)}(u-k)^2\right)^{1+\varepsilon}.\end{align*}
In summary, we obtain,
for any $0<r<R< 1$ and $0\le k<h$,
$$\|(u-h)^+\|_{L^2(G_r)}\le
C\left\{\frac{1}{R-r}
+\frac{1}{h-k}\|f\|_{L^q(G_1)}\right\}\frac{1}{(h-k)^\varepsilon}
\|(u-k)^+\|_{L^2(G_R)}^{1+\varepsilon}.$$
For any $\theta\in (0,1)$, a standard iteration yields
$$\sup_{G_{\theta}} u^+\le C\left\{\|u^+\|_{L^2(G_1)}+
\|f\|_{L^q(G_1)}\right\}.
$$
This is the desired result.
\end{proof}

Next, we prove a lower bound for positive supersolutions.

\begin{lemma}\label{moser2a}
Let \eqref{eq-Deg1}, \eqref{eq-Deg2} and
\eqref{eq-Deg3} be assumed and $f\in L^q(G_1)$, for some $q>3/2$. Suppose
$u\in C^1(\bar G_1)$ is positive and satisfies
$$\partial_i(a_{ij}u_j)+b_iu_i\le f\quad\text{weakly in }G_1.$$
Then, for any $\varepsilon\in (0,1)$, there exist constants
$\delta>0$ and $C>1$, depending only on $q$,
$\varepsilon$ and $C_*$ in \eqref{eq-Deg1} and
\eqref{eq-Deg3}, such that, if
$$|\{(x,t)\in G_1:\, u(s,t)\ge \frac12\}|\ge \varepsilon|G_1|, $$
and
$$\|f\|_{L^q(G_1)}\le \delta,$$
then,
\begin{equation}\label{413z}
\inf_{G_{1/2}}u\ge \frac1C.\end{equation}
\end{lemma}

\begin{proof} Let $\varphi$ be a nonnegative smooth cutoff function with support
in $G_1\cup \{(s,0): |s|<1\}$. Then,
\begin{equation}\label{eq-SuperSolu0}
\int a_{ij}u_i\varphi_j-\int b_iu_i\varphi\ge -\int
f\varphi.\end{equation} If $f$ is not identically zero, we take
$\delta=\|f\|_{L^q(B_1)}$. Otherwise, we take an arbitrary
$\delta>0$.\\
Now by replacing $\varphi$ by $\varphi/(u+\delta)$ in
\eqref{eq-SuperSolu0}, we have
$$-\int a_{ij}\frac{u_iu_j}{(u+\delta)^2}\varphi
+\int a_{ij}\frac{u_i}{u+\delta}\varphi_j -\int
b_i\frac{u_i}{u+\delta}\varphi\ge -\int
\frac{f}{u+\delta}\varphi.$$  Then setting
$$v=\log\frac{1}{u+\delta},$$
we get
$$-\int a_{ij}v_iv_j\varphi-\int a_{ij}v_i\varphi_j+\int b_iv_i\varphi\ge
-\int \frac{f}{u+\delta}\varphi.$$
In particular, $v$ satisfies
$$\int a_{ij}v_i\varphi_j-\int b_iv_i\varphi\le
\int \frac{f}{u+\delta}\varphi.$$
The choice of $\delta$ implies $\|f/\delta\|_{L^q(G_1)}\le 1$.
Then, for any $\theta\in (1/2, 1)$, Lemma \ref{moser2} implies
\begin{equation}\label{eq-SuperSolu1}
\sup_{G_{1/2}}(v^+)^2\le C\left\{\int_{G_\theta}(v^+)^2+1\right\},\end{equation}
where $C$ is a positive constant depending only on $q$, $\theta$ and $C_*$ in
\eqref{eq-Deg1} and
\eqref{eq-Deg3}.

Now, replace $\varphi$ in \eqref{eq-SuperSolu0} by
$$\left(\frac{1}{u+\delta}-1\right)^+\varphi^2.$$
Then,
we have
\begin{align*}\int a_{ij}\partial_iv^+\partial_jv^+\varphi^2
&\le -2\int \varphi(1-u-\delta)^+a_{ij}\partial_iv^+\varphi_j\\
&\qquad +
\int \varphi^2(1-u-\delta)^+b_i\partial_iv^+
+\int \varphi^2\frac{f}{u+\delta}(1-u-\delta)^+.\end{align*}
By writing $b_2=b_{21}+b_{22}$, we now consider the $b_{21}$-term and write
\begin{align*}\int_{G_1} \varphi^2(1-u-\delta)^+b_{21}\partial_tv^+&=
\int_{G_1}\varphi^2b_{21}\partial_t\left[\left(\log\frac{1}{u+\delta}\right)^+-
(1-u-\delta)^+\right]\\
& =-\int_{G_1}
\partial_t(\varphi^2b_{21})\left[\left(\log\frac{1}{u+\delta}\right)^+-
(1-u-\delta)^+\right]\\
&\qquad+\int_{\partial G_1} \varphi^2b_{21}
\left[\left(\log\frac{1}{u+\delta}\right)^+-
(1-u-\delta)^+\right]\nu_t.\end{align*} Note that $\varphi=0$ on
$\partial G_1\setminus\{t=0\}$ and $\nu_t=-1$ and $b_{21}\ge 0$ on
$\{t=0\}$. Also, for $u+\delta<1$,
\begin{equation}\label{eq-Super3}
\left(\log\frac{1}{u+\delta}\right)^+> (1-u-\delta)^+.
\end{equation} Therefore,
$$\int_{G_1} \varphi^2(1-u-\delta)^+b_{21}\partial_tv^+\le
\int_{G_1} |\partial_t(\varphi^2b_{21})|\left(\log\frac{1}{u+\delta}\right)^+,$$
and hence
\begin{align*}\int a_{ij}\partial_iv^+\partial_jv^+\varphi^2
&\le -2\int \varphi(1-u-\delta)^+a_{ij}\partial_iv^+\varphi_j\\
&\qquad+\int \varphi^2(1-u-\delta)^+(b_1\partial_sv^++b_{22}\partial_tv^+)\\
&\qquad +
\int |\partial_t(\varphi^2b_{21})|v^+
+\int \varphi^2\frac{f}{u+\delta}(1-u-\delta)^+.\end{align*}
By proceeding as in the proof of Lemma \ref{garding1}, we have
$$\int \varphi^2\big(t(\partial_tv^+)^2+(\partial_sv^+)^2\big)\le C
\left\{\int(t\varphi_t^2+\varphi_s^2+\varphi^2)
+\int(\varphi+|\varphi_t|)\varphi v^++\int\varphi^2
\frac{f}{\delta}\right\}.$$
The choice of $\delta$ implies $\|f/\delta\|_{L^q(G_1)}\le 1$. Hence,
for any $\theta_1<\theta_2<1$, we take
$\varphi=1$ in $G_{\theta_1}$ and $\varphi=0$ in $G_1\setminus G_{\theta_2}$. Then,
for any $\tau\in (0,1)$ to be determined, we have
\begin{equation}\label{eq-SuperSolu2a}\int_{G_{\theta_1}}
\big(t(\partial_tv^+)^2+(\partial_sv^+)^2\big)\le
\frac{C_\tau}{(\theta_2-\theta_1)^2}+\tau\int_{G_{\theta_2}}(v^+)^2.\end{equation}
Note \begin{align*}
&\, |\{(s,t)\in G_{\theta_1}:\, v^+=0\}|\\
\ge &\, |\{(s,t)\in
G_1:\, u+\delta\ge 1\}|-|G_1|+|G_{\theta_1}|\\
\ge&\, |G_{\theta_1}|-(1-\varepsilon)|G_1|
=\left(1-\frac{1-\varepsilon}{\theta_1^{\frac 32}}\right)|G_{\theta_1}|
\ge \frac12\varepsilon|G_{\theta_1}|,
\end{align*}
by taking $\theta_1$ such that
$$\theta_0\equiv
\max\left\{\frac12, \frac{1-\varepsilon}{1-\frac{\varepsilon}{2}}\right\}
<\theta_1<1.$$
Then Lemma \ref{weight1}(2) implies
\begin{equation}\label{eq-SuperSolu2b}
\int_{G_{\theta_1}}(v^+)^2\le C\int_{G_{\theta_1}}\big(t(v^+_t)^2
+(v_s^+)^2\big)\ \text { for all }\theta_1 \ge \theta_0.
\end{equation}
It must be emphasized that $C$ in \eqref{eq-SuperSolu2b} depends
on $\varepsilon$ through $\theta_0$, and is independent of
$\theta_1$. By combining \eqref{eq-SuperSolu2a} and
\eqref{eq-SuperSolu2b}, we have
$$\int_{G_{\theta_1}}(v^+)^2\le \frac{C_\tau}{(\theta_2-\theta_1)^2}+C\tau\int_{G_{\theta_2}}(v^+)^2.$$
Now choose $\tau$ such that $C\tau=1/2$. We obtain, for any
$\theta_0<\theta_1<\theta_2<1$,
$$\int_{G_{\theta_1}}(v^+)^2\le \frac{C_\tau}{(\theta_2-\theta_1)^2}+\frac12\int_{G_{\theta_2}}(v^+)^2.$$
A standard iteration yields, for any $\theta_0<\theta<1$,
\begin{equation}\label{eq-SuperSolu2}
\int_{G_{\theta}}(v^+)^2\le \frac{C}{(1-\theta)^2}.
\end{equation}

By combining \eqref{eq-SuperSolu1} and \eqref{eq-SuperSolu2}
and fixing a $\theta\in (\theta_0,1)$, we obtain
$$\sup_{G_{1/2}}(v^+)^2\le C,$$
and hence
$$\inf_{G_{1/2}}u+\delta\ge e^{-C}.$$
We note that the constant $C$ above is independent of $\delta$.
If $f\equiv0$, we simply let $\delta\to0$.
Otherwise, by taking
$\delta=e^{-C}/2$, we have the desired estimate.
\end{proof}

Now, we are ready to prove an estimate of  boundary H\"{o}lder norms.

\begin{theorem}\label{moser2b}
Let \eqref{eq-Deg1}, \eqref{eq-Deg2} and
\eqref{eq-Deg3} be assumed and $f\in L^q(G_1)$, for some $q>3/2$. Suppose
$u\in C^1(\bar G_1)$ satisfies
$$\partial_i(a_{ij}u_j)+b_iu_i= f\quad\text{in }G_1.$$
Then, for some $\alpha\in (0,1)$,
\begin{equation}\label{418z}
|u(s,t)-u(s,0)|\le C\left(\sup_{G_1}|u|+\|f\|_{L^q(G_1)}\right)t^\alpha
\quad\text{for any }(s,t)\in G_{1/2},\end{equation}
where $\alpha$ and $C$ are positive constants depending only on
$q$ and $C_*$
in \eqref{eq-Deg1} and
\eqref{eq-Deg3}.
\end{theorem}

\begin{proof} Set, for any $r\le 1$,
$$M(r)=\sup_{G_r}u,\quad m(r)=\inf_{G_r}u,$$
and
$$
\omega(r)=M(r)-m(r).$$
We now claim, for any $r\le 1$,
\begin{equation}\label{eq-Holder1}
\omega\left(\frac r2\right)\le \gamma\omega(r)+Cr^{1-\frac{3}{2q}}\|f\|_{L^q(G_r)},
\end{equation}
where $\gamma\in (0,1)$ and $C>1$ are constants depending only on
$q$ and $C_*$
in \eqref{eq-Deg1} and
\eqref{eq-Deg3}.
By a simple iteration, we have, for any $r\le 1/2$,
$$\omega(r)\le  Cr^\alpha\left\{\omega(1)+\|f\|_{L^q(G_1)}\right\},$$
where $\alpha\in (0,1)$ and $C>1$ are constants depending only on
$q$ and $C_*$
in \eqref{eq-Deg1} and
\eqref{eq-Deg3}.

We now prove \eqref{eq-Holder1} for $r=1$. The general case follows from
a simple scaling. Let $\varepsilon=1/2$ and $\delta$ be determined as in
Lemma \ref{moser2a}. If
$$\delta \omega(1)\le \|f\|_{L^q(G_1)},$$
then,
\begin{equation}\label{eq-Holder2}
\omega\left(\frac 12\right)\le \omega(1)\le \frac1\delta\|f\|_{L^q(G_1)}.
\end{equation}
Next, we assume
$$\|f\|_{L^q(G_1)}\le \delta \omega(1).$$
We note that $u/\omega(1)$ satisfies
$$\partial_i\left(a_{ij}\partial_j\left(\frac{u}{\omega(1)}\right)\right)
+b_i\partial_i\left(\frac{u}{\omega(1)}\right)
=\frac{f}{\omega(1)}\quad\text{in }G_1.$$ Hence
$$\left\|\frac f{\omega(1)}\right\|_{L^q(G_1)}\le \delta$$
by the previous assumption. We consider the following two
cases:
\begin{equation}\label{eq-caseA}
|\{(s,t)\in G_1:\, \frac {u-m(1)}{M(1)-m(1)}\ge
\frac12\}|\ge \frac 12|G_1|,
\end{equation} and
\begin{equation}\label{eq-caseB}
|\{(s,t)\in G_1:\, \frac {M(1)-u}{M(1)-m(1)}\ge
\frac12\}|\ge \frac 12|G_1|.
\end{equation}
If \eqref{eq-caseA} holds, we  apply Lemma \ref{moser2a} to
$(u-m(1))/(M(1)-m(1))$ and get
$$m\left(\frac12\right)-m(1)\ge \frac 1C(M(1)-m(1)). $$
If \eqref{eq-caseB} holds, we  apply Lemma \ref{moser2a} to
$(M(1)-u)/(M(1)-m(1))$ and get
$$M(1)-M\left(\frac12 \right)\ge \frac 1C(M(1)-m(1)). $$
Since $m(1/2)\ge m(1)$ and $M(1/2)\le M(1)$, we have in both cases
$$M\left(\frac12\right)-m\left(\frac12\right)\le \left(1-\frac1C\right)(M(1)-m(1)),$$
and hence
\begin{equation}\label{eq-Holder3}
\omega\left(\frac 12\right)\le \gamma\omega(1),
\end{equation}
for some constant $\gamma\in (0,1)$. We have
\eqref{eq-Holder1} by combining \eqref{eq-Holder2} and \eqref{eq-Holder3}. \end{proof}

Now, we prove two estimates of  $N$. The first  concerns an energy estimate
of $N$ and the second concerns a boundary H\"{o}lder estimate of $N$.

\begin{theorem}\label{moser3}
Let $\Sigma$ be an Alexandrov-Nirenberg
surface in $\mathbb R^3$ of class $C^5$ and $\sigma$ be a connected component of
$\partial\Sigma$.
Then, in the geodesic coordinates as in (\ref{geod1}) and
(\ref{geod2}),
\begin{equation}\label{418q}
\int_0^{\frac12}\int_{0}^{2\pi}\big(tN_t^2+N_s^2)dsdt\le C,
\end{equation}
and
\begin{equation}\label{418}
|N(s,t)-N(s,0)|\le Ct^\alpha\quad\text {for any }t\le 1,
\end{equation}
where $\alpha\in (0,1)$ and $C>0$ are constants depending only on the
quantities in (\ref{maincond}).
\end{theorem}

\begin{proof} Set
$$u=\frac{1}{N^2}.$$
By \eqref{eq7/N} and with slightly different notations, $u$ satisfies
$$\partial_i(a^{ij}\partial_ju)+b^i\partial_iu=f,$$
where
\begin{align*}&a^{11}=N^2,\quad a^{12}=-MN, \quad a^{22}=LN,\\
&b^2=2B^2K_t+3BB_tK-\frac{5B_t}{B}M^2+\frac{B_s}{B}MN,
\end{align*}
and $b^1$ and $f$ are bounded by  Lemma \ref{Remark1.13}.
We now verify \eqref{eq-Deg1}, \eqref{eq-Deg2} and \eqref{eq-Deg3}.

By Lemma \ref{Remark1.13} and Lemma \ref{Lemma1.17}, we have
$$\frac{1}{C}\le N\le C, \quad |M|\le C\sqrt{t}, \quad L\le Ct.$$
Then, $b^2=2B^2K_t+\tilde b^2t$, where $\tilde b^2$  is a bounded
function, and hence \eqref{eq-Deg2} and \eqref{eq-Deg3} hold.
Moreover,
$$N^2\xi_1^2-2MN\xi_1\xi_2+LN\xi_2^2\le C(\xi_1^2+t\xi_2^2).$$
Next, by $LN=M^2+KB^2$, we have
$$N^2\xi_1^2-2MN\xi_1\xi_2+LN\xi_2^2=(N\xi_1-M\xi_2)^2+KB^2\xi_2^2,$$
and hence, by choosing $c$ small,
\begin{align*}
cN^2(\xi_1^2+t\xi_2^2)&\le 2c(N\xi_1-M\xi_2)^2+2cM^2\xi_2^2+cN^2t\xi_2^2\\
&=2c(N\xi_1-M\xi_2)^2+\big(2cM^2+cN^2t\big)\xi_2^2\\
&\le (N\xi_1-M\xi_2)^2+KB^2\xi_2^2.\end{align*} Therefore,
\eqref{eq-Deg1} is satisfied. By Lemma \ref{garding1}, we obtain
\begin{align*}&
\int_0^{\frac12}\int_{0}^{2\pi}\left(t\left(\partial_t(N^{-2})\right)^2
+\left(\partial_s(N^{-2})\right)^2\right)dsdt\\
&\qquad \le C \int_0^1\int_0^{2\pi}(N^{-4}+f^2)dsdt\le
C.\end{align*} We point out that $u$ is periodic in
$s$. Hence, we can take $\varphi$ as a cutoff function of $t$
near $t=1$. We then have the desired result by the boundedness of
$N$.

By Theorem \ref{moser2b}, we obtain,  for any $0<t\le 1/2$,
$$\left|\frac 1{N(s,t)^2}-\frac 1{N(s,0)^2}\right|\le Ct^\alpha\left(\sup_{t\in (0,1)}\left|\frac
1{N^2}\right|+\sup_{t\in (0,1)}|f|\right)\le C_1 t^\alpha,$$
and hence
$$
|N(s,t)-N(s,0)|=\left|\frac 1{N(s,t)^2}-\frac 1{N(s,0)^2}\right|\frac {N(s,t)^2N(s,0)^2}{N(s,t)+N(s,0)}
\le C_1 t^\alpha.$$
As for $t\in [1/2,1] $, (\ref{418}) follows
immediately from the boundedness of $N$.  We thus have the desired result. \end{proof}

\section{Lipschitz Estimates near Boundary}\label{sec-BoundaryLipschitz}

In this section, we derive the Lipschitz norms of
the second fundamental form near boundary.  Lemma \ref{Lemma1.17} and
Theorem \ref{moser3} are not  enough for
$C^{2,\alpha}$-estimates. We need
a result stronger than Theorem \ref{moser3} for $N(s,t)-N(s,0)$ and
a result stronger  than Lemma
\ref{Lemma1.17} for $M(s,t)$. We will employ blowup techniques in this section.

In the proof of the next result, we will use Theorem \ref{rem1} to conclude the smoothness
of solutions to a limit equation.

\begin{theorem}\label{N2}
Let $\Sigma$ be an Alexandrov-Nirenberg
surface in $\mathbb R^3$ of class $C^5$ and $\sigma$ be a connected component of
$\partial\Sigma$.
Then, in the geodesic coordinates as in (\ref{geod1}) and
(\ref{geod2}),
\begin{equation}\label{418a}
|N(s,t)-N(s,0)|\le Ct\quad\text {for any }t\le 1,
\end{equation}
where $C$ is a positive constant depending only on the
quantities in (\ref{maincond}).
\end{theorem}


\begin{proof} Set
$$u=\frac{1}{N^2}-\frac{1}{N^2(s,0)}.$$
By \eqref{eq8/N}, \eqref{eq9/N} and \eqref{eq10/Na},  and with slightly
different notations, $u$ satisfies
$$\partial_i(a^{ij}\partial_ju)+{b}^i\partial_iu=f,$$
where
\begin{align*}&a^{11}=N^2,\quad a^{12}=-MN, \quad a^{22}=LN,\\
&{b}^2=2B^2K_t+3BB_tK-\frac{5B_t}{B}M^2+\frac{B_s}{B}MN
+\frac12MN^3\partial_s\left(\frac{1}{N^2(s,0)}\right),
\end{align*} and $b^1$ and $f$ are bounded.

We now prove \eqref{418a} by contradiction. If it was
false, then there would exist a sequence of Alexandrov-Nirenberg surfaces
$\Sigma_k$, with their induced metrics $g_k$ in $\bar D$,
such that, in the geodesic
coordinates as in (\ref{geod1}) and (\ref{geod2}),
$$|g_k|_{C^4}, \left(\inf\partial_tK_k(s,0)\right)^{-1}\text { and }
\left(\inf\partial_tB_k(s,0)\right)^{-1}\text{ are uniformly bounded},$$
and
\begin{equation}
\theta_k\equiv\sup\left\{\frac{|N_k(s,t)-N_{k}(s,0)|}{t}\right\}
\rightarrow \infty\quad \text{as } k\rightarrow \infty.
\end{equation}
Set $\Omega_1=\{(s,t):\, s\in [0,2\pi],\, t\in (0,1)\}$.
We also assume $g_k\to g$ in $C^3(\bar \Omega_1)$ for some smooth
metrics $g$ on $\bar \Omega_1$.
Let
$(s_k,t_k)$ be a point  such that $t_k>0$ and
$$\frac{|N_k(s_k,t_k)-N_{k}(s_k,0)|}{t_k}\ge
\frac12\theta_k.$$ Without loss of generality, we assume
$(s_k,t_k)\rightarrow (0,0)$. By Lemma \ref{Remark1.13}
and Lemma \ref{Lemma1.17}, we have, for any $k\ge 1$,
$$\frac1C\le N_k\le C, \quad |M_k|\le C\sqrt t,$$
and hence
$$t_k\theta_k\le C,$$
where $C$ is a positive constant under control. Furthermore, Theorem \ref{moser3}
implies
\begin{equation}
\label{325}
|N_k(s,t)-N_k(s,0)|\le Ct^\alpha,\end{equation}
where $\alpha\in (0,1)$ is a constant under control.

Set
$$u_k(s,t)=\frac{1}{N_k^2(s,t)}-\frac{1}{N_k^2(s,0)}.$$
Then,
\begin{equation}\label{334}\partial_i(a_k^{ij}\partial_ju_k)
+b_k^i\partial_iu_k=f_k,\end{equation}
where $a_{k}^{ij}, b_k^i$ and $f_k$ are uniformly bounded, independent of $k$.
Consider the transform
\begin{equation}\label{4231}
x=\frac{s-s_k}{\sqrt{t_k}},\quad y=\frac{t}{t_k},\end{equation}
and set
$$w_k(x,y)=u_k(s,t),\quad  \bar{w}_k(x,y)=\frac{1}{\theta_kt_k}w_k(x,y).$$
Then,
\begin{equation}\label{3191}
|\bar{w}_k|\leq Cy,\quad |\bar{w}_k(0,1)|\ge \frac
1C,\end{equation} for some constant $C$ under control.
In  the new coordinates $(x,y)$, $w_k$ and $\bar w_k$
satisfy, with $(\partial_1, \partial_2)=(\partial_x, \partial_y)$,
\begin{equation}\label{304a}\partial_i(\widetilde a_k^{ij}\partial_j w_k)+
\widetilde b_k^i\partial_i w_k=t_kf_k,\end{equation}
and
\begin{equation}\label{304}\partial_i(\widetilde a_k^{ij}\partial_j\bar w_k)+
\widetilde b_k^i\partial_i\bar w_k=\theta_k^{-1}f_k,\end{equation}
where
$$\widetilde a^{11}_k=a^{11}_k, \quad
\widetilde a^{12}_k=\frac{a^{12}_k}{\sqrt{t_k}}, \quad
\widetilde a^{22}_k=\frac{a^{22}_k}{t_k},$$
and
$$\widetilde b^1_k=\sqrt{t_k}b^1_k, \quad \widetilde b^2_k=b^2_k.$$
In particular,
$$\widetilde a^{11}_k=N_k^2, \quad
\widetilde a^{12}_k=-\frac{1}{\sqrt{t_k}}M_kN_k, \quad
\widetilde a^{22}_k=\frac{1}{t_k}L_kN_k.$$
Hence,
$$C^{-1}(\xi_1^2+y\xi_2^2)\le \widetilde a^{ij}_k\xi_i\xi_j\le C(\xi_1^2+y\xi_2^2)
\quad\text{for any }\xi\in\mathbb R^2,$$
for some constant $C$ under control.
Therefore, the principle part of (\ref{304}) is an elliptic
equation of divergence form with bounded measurable coefficients
in the region $\{y>\delta\}$. Then, for any  $\delta, R>0$,  there is a positive
$\beta\in (0,1)$, depending only on $\delta$ and $R$, such that
\begin{equation}\label{324z} |\bar w_k|_{C^\beta(B_R\cap\{y>\delta\})}\le C_{R\delta}.\end{equation}

Next, let $\psi$ be a cut-off
function in $\mathbb R^2$. Then we claim
\begin{align}
\label{324}\int_{\mathbb R^2_+}
\psi^2\left(y(\partial_y\bar{w}_k)^2+(\partial_x\bar{w}_k)^2\right)&\le
C_\psi,\\
\label{324a}\int_{\mathbb R^2_+}
\psi^2\left(y(\partial_y{w}_k)^2+(\partial_x{w}_k)^2\right)&\le
C_\psi t_k^{2\alpha},
\end{align}
where $\alpha$ is the constant as in Theorem \ref{moser3} and
$C_{\psi}$ is a positive constant depending only on $\psi$ and other
quantities under control.
To see this, we first note by Lemma \ref{garding1}
\begin{align}\label{324q}\begin{split}
&\int
\psi^2\left(y(\partial_y\bar{w}_k)^2+(\partial_x\bar{w}_k)^2\right)\\
&\qquad\leq C\int
(\psi^2+y\psi_y^2+\psi_x^2+|\psi_y|)\bar{w}_k^2+C\theta_k^{-2}\int
\psi^2|f_k|^2.
\end{split}\end{align}
This implies \eqref{324} easily by \eqref{3191}.
Similarly,
$w_k$ satisfies
\begin{align*}
&\int
\psi^2\left(y(\partial_y{w}_k)^2+(\partial_x{w}_k)^2\right)\\
&\qquad\leq C\int
(\psi^2+y\psi_y^2+\psi_x^2+|\psi_y|){w}_k^2+Ct_k^{2}\int
\psi^2|f_k|^2.
\end{align*}
By (\ref{325}), we have
$$|w_k(x,y)|\le C|N_k(s,t)-N_k(s,0)|\le Ct^\alpha=Ct_k^\alpha y^\alpha,$$
and hence \eqref{324a} follows.

Next, we claim
\begin{equation}\label{324bb}\int \psi^2\frac { M^2_k}{t_k}\rightarrow 0\quad\text {as
}k\rightarrow \infty.\end{equation}
To prove this, we note that  (\ref{GC2}) implies
$$\partial_tM_k+\frac 12N_k^3\partial_s\left(\frac 1{N^2_k}-\frac 1{N^2_k(0,s)}\right)
=-\frac{B_t}{B}M_k-\frac 12N_k^3\partial_s\left(\frac
1{N^2_k(0,s)}\right)\equiv h_k,$$ and hence
\begin{equation}\label{332}
\partial_y M_k+\frac 12\sqrt {t_k}N_k^3\partial_x w_k =t_k{h}_k.
\end{equation}
Then,
\begin{equation*}\label{316} \int
\psi^2(\partial_yM_k)^2\leq Ct_k\int
\psi^2(\partial_xw_k)^2+Ct_k^{2}\int \psi^2h_k^2.
\end{equation*}
In view of the fact that $M_k(x,0)=0$, we have, for arbitrary
$ r,T>0$ and any $\psi\in C^1_0(\mathbb R^2)$ with $\psi=1$ on $[-r, r]\times [0, T]$,
\begin{align*}\label{315}
\int_{0}^T\int_{-r}^{r} \frac {| M_k|^2}{t_k}dxdy&\le
T^2\int_{0}^T\int_{-r}^r \frac {|\partial_y
M_k|^2}{t_k}dxdy\\
&\leq C\int \psi^2(\partial_xw_k)^2+Ct_k\int
\psi^2h_k^2\rightarrow 0 \quad\text{as } k\rightarrow \infty,
\end{align*}
where we used \eqref{324a} in the final step. This finishes the proof of
\eqref{324bb}. In terms of  coefficients, we have
\begin{equation}\label{324b}\int \psi^2(\widetilde a^{12}_k)^2\rightarrow 0\quad\text {as
}k\rightarrow \infty.\end{equation}

In view of \eqref{324z}, we can find a subsequence of $\{\bar w_k\}$, still
denoted by $\bar w_k$, such that,
$$\bar w_k\rightarrow w\quad\text{locally uniformly in }\mathbb R^2_+,$$
for some $w\in C(\mathbb R^2_+)$.
By (\ref{3191}), we have
$$|w(x,y)|\le Cy\quad\text{in }\mathbb R^2_+,$$
and
\begin{equation}\label{3191p}|w(0,1)|\ge\frac 1{C}.\end{equation}
The former estimate implies $w\in C(\bar
{\mathbb R}^2_+)$ and
$$w(x,0)=0.$$
In the following, we prove that $w$ satisfies
\begin{equation}\label{weak1}
yw_{yy}+\frac {N^2(0)}{K_t(0)}w_{xx}+3w_y=0 \quad\text{in }\mathbb R^2_+.
\end{equation}
Indeed, for any cut-off function $\psi\in C_c^{\infty}(\mathbb R^2_+)$,
multiplying both sides of (\ref{304}) by $\psi$ and integrating by
parts,  we have
\begin{equation}\label{327}
\int\big(\widetilde a^{ij}_k\partial_i\bar w_k\partial_j\psi
-\widetilde b^i_k\partial_i\bar w_k\psi\big)
=-\int\psi\theta_k^{-1}f_k.\end{equation}
Since $f_k$ and $b^1_k$ are bounded, we have
$$\int\psi\theta_k^{-1}f_k\to 0,$$
and, by \eqref{324},
$$\int \psi\widetilde b^1_k\partial_x\bar w_k=\sqrt{t_k}\int \psi b^1_k\partial_x\bar w_k
\rightarrow 0,$$
as $k\rightarrow \infty.$
Next, since
$$\widetilde b^2_k=b^2_k=2\partial_tK_k(s,0)+O(\sqrt{t})=2\partial_tK_k(s,0)+\sqrt{t_k}O(\sqrt{y}),$$
then
$$\left|\int\psi\widetilde b^2_k\partial_y\bar w_k-2\int\partial_tK_k(s,0)\psi\partial_y\bar w_k\right|
\le C\sqrt{t_k}\int\psi\sqrt{y}|\partial_y\bar w_k|,$$
and hence, by \eqref{324},
$$\int \psi \widetilde b^2_k\partial_y\bar w_k\rightarrow 2\int
\psi K_t(0)w_y,
$$
as $k\rightarrow \infty$.
By \eqref{324} and \eqref{324b}, we have
$$
\int |\psi_y\widetilde a^{12}_k\partial_x\bar w_k|\le
\left(\int |\psi_y|(\widetilde a^{12}_k)^2\right)^{\frac12}\left(\int
|\psi_y| |
\partial_x\bar w_k|^2\right)^{\frac12}\rightarrow 0,$$
as $k\rightarrow \infty.$
Next, by \eqref{325} and $\widetilde a^{11}_k=N_k^2$, we have
\begin{equation*}
\int |\psi_x(\widetilde a^{11}_k-\widetilde a^{11}_k(0))\partial_x\bar w_k|
\le Ct_k^\alpha\left(\int |\psi_x|\right)^{\frac12}\left(\int |\psi_x|
|\partial_x\bar w_k|^2\right)^{\frac12}
\rightarrow 0, \end{equation*}
or
$$\int \psi_x\widetilde a^{11}_k\partial_x\bar w_k\to \int N^2(0)\psi_xw_x,$$
as $k\rightarrow \infty$. We note that $N_k(s,0)$ is intrinsically determined by
Lemma \ref{Lemma1.1}.
For the $\widetilde a^{22}_k$ term, we note
$$
\int\psi_y\widetilde a^{22}_k\partial_y\bar w_k=\int\psi_y\frac {N_kL_k}{t_k}\partial_y\bar w_k
=\int\psi_y\frac{M^2_k}{t_k}\partial_y\bar w_k+\int\psi_y\frac{K_kB_k^2}{t_k} \partial_y\bar w_k.$$
By writing
$$\frac{M_k^2}{t_k}=\frac{M_k}{\sqrt{t_k}}\cdot \frac{M_k}{\sqrt{t_k}}\le C\frac{\sqrt t}{\sqrt{t_k}}\cdot
\frac{|M_k|}{\sqrt{t_k}}\le C\sqrt{y}\frac{|M_k|}{\sqrt{t_k}},$$
we have, by \eqref{324} and \eqref{324bb}
$$\left|\int\psi_y\frac{M^2_k}{t_k}\partial_y\bar w_k\right|\le C\left(\int|\psi_y|\frac{M_k^2}{t_k}\right)^{\frac12}
\left(\int |\psi_y|y|\partial_y\bar w_k|^2\right)^{\frac12}\to 0,$$
as $k\to\infty$. Moreover,
\begin{align*}K_k(s,t)&=K_k(s_k+\sqrt{t_k}x, t_ky)-K(s_k+\sqrt{t_k}x, 0)\\
&=
t_ky\int_0^1\partial_tK_k(s_k+\sqrt{t_k}x, t_ky\tau)d\tau.\end{align*}
Therefore,
$$\frac{K_k}{t_k}\to yK_t(0)\quad\text{locally uniformly in }\mathbb R^2_+,$$
and hence
$$
\int\psi_y\widetilde a^{22}_k\partial_y\bar w_k\to\int\psi_yK_t(0)yw_y,$$
as $k\to\infty$.
Finally,  by passing to the limit in
(\ref{327}), we get
$$\int \big(
K_t(0)y\psi_yw_y+N^2(0)\psi_xw_x- 2K_t(0)\psi w_y\big)=0.$$ This
is simply the equation (\ref{weak1}) in the weak sense after
multiplying both sides by $1/K_t(0)$. Therefore, applying Theorem
\ref{rem1} to (\ref{weak1}), we  conclude that its solution $w\in
C^{\infty}(\{y\ge 0\})$. Moreover, $w$ is analytic in a
neighborhood of $0$ as shown in \cite{L2}; namely, $w$ can be
expanded in terms of  a Taylor series in $B_r(0)\cap \{y\ge 0\}$
for some positive constant $r$. Now by (\ref{weak1}) and $w=0$ on
$y=0$,  we get $\partial_y^kw=0$ on $y=0$ for all $k=1,2,\cdots$.
Therefore, $w\equiv 0$ in $\mathbb R^2_+$, which contradicts
\eqref{3191p}. This ends the proof for the present theorem.
\end{proof}

We now improve the estimate for $M$ in Lemma \ref{Lemma1.17}.

\begin{theorem}\label{M2}
Let $\Sigma$ be an Alexandrov-Nirenberg
surface in $\mathbb R^3$ of class $C^5$ and $\sigma$ be a connected component of
$\partial\Sigma$.
Then, in the geodesic coordinates as
in (\ref{geod1}) and (\ref{geod2}),
\begin{equation}\label{418b}
|M(s,t)|\le Ct\quad\text {for any }t\le 1,
\end{equation}
where $C$ is a positive constant depending only on the
quantities in (\ref{maincond}).
\end{theorem}

\begin{proof} By \eqref{eq2M} and (\ref{mequation1p}),
and with slightly different notations, $M$ satisfies
$$\partial_i(a^{ij}\partial_jM)+b^i\partial_iM+ca^{ij}\partial_iM\partial_jM=f,$$
where
\begin{equation*}
a^{11}=N, \quad a^{12}=-M, \quad a^{22}=L, \quad c=-\frac{2M}{NL},
\end{equation*}and $b_i$ and $f$ are bounded by Remark \ref{rmk-Coeff-M} and Theorem \ref{N2}.

We now prove \eqref{418b} by contradiction. If it
was false, then there would exist a sequence of Alexandrov-Nirenberg surfaces
$\Sigma_k$, with their  induced metrics $g_k$ in $\bar D$,
such that, in the geodesic coordinates as
in (\ref{geod1}) and (\ref{geod2}),
$$|g_k|_{C^4}, \left(\inf\partial_tK_k(s,0)\right)^{-1}
\text { and }\left(\inf\partial_tB_k(s,0)\right)^{-1}\text{ are uniformly bounded},$$
and
\begin{equation}
\theta_k\equiv\sup\left\{\frac{|M_k(s,t)|}{t}\right\}\rightarrow \infty\quad \text{as } k\rightarrow \infty.
\end{equation}
Set $\Omega_1=\{(s,t):\, s\in [0,2\pi],\, t\in (0,1)\}$.
We also assume $g_k\to g$ in $C^3(\bar \Omega_1)$ for some smooth
metrics $g$ on $\bar \Omega_1$.
Let $(s_k,t_k)$ be a point such that $t_k>0$
$$\frac{|M_k(s_k,t_k)|}{t_k}\ge \frac12\theta_k.$$
Without loss of generality, we assume
$(s_k,t_k)\rightarrow (0,0)$.
By Corollary \ref{Remark1.13}
and Lemma \ref{Lemma1.17}, we have, for any $k\ge 1$,
$$\frac1C\le N_k\le C, \quad |M_k|\le C\sqrt t,$$
and hence
\begin{equation}\label{331}\sqrt{t_k}\theta_k\le C,\end{equation}
where $C$ is a positive constant under control.

Consider the transform
\begin{equation}\label{4231a}
x=\frac{s-s_k}{\sqrt{t_k}},\quad y=\frac{t}{t_k},\end{equation}
and set
$$\bar M_k(x,y)=\frac{1}{t_k\theta_k}M_k(s,t),$$
and
$$w_k(x,y)=\frac{1}{N^2_k(s,t)}-\frac{1}{N^2_k(s,0)},\quad  \bar{w}_k(x,y)=\frac{w_k}{\theta_kt_k}.$$
Then,
\begin{equation}\label{3191a}
|\bar{M}_k(0,1)|\ge \frac
12.\end{equation}
Moreover, by Theorem \ref{N2},
\begin{equation}\label{3191b}
|\bar{w}_k|\leq  C\theta_k^{-1}y.\end{equation}

In the original coordinates $(s,t)$,  $M_k$ satisfies
$$\partial_i(a_k^{ij}\partial_jM_k)+b_k^i\partial_iM_k
+c_ka_k^{ij}\partial_iM_k\partial_jM_k=f_k,$$
where
$$
a^{11}_k=N_k, \quad a^{12}_k=-M_k, \quad a^{22}_k=L_k,
\quad c_k=-\frac{2M_k}{N_kL_k},$$
and $b^i$ and $f_k$ are uniformly bounded, independent of $k$.
Then in  the new coordinates $(x,y)$, $\bar M_k$ satisfies
\begin{equation}\label{mequation11}
\partial_i(\widetilde{a}_k^{ij}\partial_j\bar M_k)+\widetilde{b}_k^i\partial_i\bar M_k
+\widetilde{c}_k\widetilde{a}_k^{ij}\partial_i\bar M_k\partial_j\bar M_k=\theta_k^{-1}f_k,
\end{equation}
where
\begin{align*}
&\widetilde a_k^{11}=a^{11}_k, \quad \widetilde a_k^{12}=\frac{a_k^{12}}{\sqrt{t_k}},
\quad \widetilde a^{22}_k=\frac{a_k^{22}}{t_k},
\quad \widetilde c_k=t_k\theta_k c_k,\\
&\widetilde b^1_k=\sqrt{t_k}b^1_k, \quad \widetilde b^2_k=b^2_k.\end{align*}
As in the proof of Theorem \ref{N2}, we have
$$C^{-1}(\xi_1^2+y\xi_2^2)\le \widetilde a^{ij}_k\xi_i\xi_j\le C(\xi_1^2+y\xi_2^2)
\quad\text{for any }\xi\in\mathbb R^2,$$
for some constant $C$ under control. Moreover, by \eqref{331},
$$|\tilde c_k|=\left|\frac{2t_k\theta_kM_k}{N_kL_k}\right|
\le \frac{Ct_k\theta_k|M_k|}{K_k}\le
\frac{Ct_k\theta_k}{\sqrt t}= \frac{C\sqrt{t_k}\theta_k}{\sqrt y}\le \frac{C}{\sqrt y}.$$
As
$y>\delta$ where $\delta>0$,  the principle part of
(\ref{mequation11}) is an elliptic equation of divergence
form with bounded measurable coefficients and the nonlinear terms of
first derivatives are quadratic. Then, for any  $\delta, R>0$,  there is a positive
$\beta\in (0,1)$, depending only on $\delta$ and $R$, such that
\begin{equation}\label{324zz} |\bar M_k|_{C^\beta(B_R\cap\{y>\delta\})}\le C_{R\delta}.\end{equation}
This follows from the H\"{o}lder estimate due to de Giorgi and Moser.
We point out that the H\"{o}lder estimate still holds even with the presence
of the quadratic nonlinear terms in first derivatives.
Then, by \eqref{3191a}, there exists an
$r_0>0$ such that
\begin{equation}\label{3191aa}|\bar M_k|\ge \frac14\quad\text{in }B_{r_0}((0,1)).\end{equation}

Let $\psi$ be a cutoff function in $\mathbb R^2$. By (\ref{324q}) and \eqref{3191b}, we obtain
\begin{equation}\begin{split}\label{305}
&\int
\psi^2\left(y(\partial_y\bar{w}_k)^2+(\partial_x\bar{w}_k)^2\right)\\
&\leq C\int
(\psi^2+y\psi_y^2+\psi_x^2+|\psi_y|)\bar{w}_k^2+C\theta_k^{-2}\int
\psi^2|f_k|^2\leq C_\psi\frac 1{\theta^2_k}.
\end{split}
\end{equation}
By (\ref{332}), we have
\begin{equation*}
\partial_y \bar M_k+\frac 12\sqrt {t_k} N_k^3\partial_x \bar w_k =\frac 1{\theta_k} h_k,
\end{equation*} and hence, by (\ref{305}),
$$\int \psi^2(\partial_y\bar M_k)^2\leq Ct_k\int
\psi^2(\partial_x\bar w_k)^2+C\frac 1{\theta^2_k}\int \psi^2h_k^2
\le C_\psi\frac 1{\theta_k^2}\rightarrow 0,$$
as $k\rightarrow \infty$. In view of the fact that $\bar M_k(x,0)=0$, we
have, for arbitrary $r,T>0$
$$
\int_{0}^T\int_{-r}^{r} |\bar M_k|^2dxdy\\
\le T^2\int_{0}^T\int_{-r}^r |\partial_y
\bar M_k|^2dxdy\rightarrow 0,$$
as $k\rightarrow \infty$.
This contradicts \eqref{3191aa} and hence completes the proof of the
present theorem.
\end{proof}

\section{Higher Order Estimates near Boundary}\label{sec-HigherOrderEstimates}

In this section, we derive estimates of higher order derivatives of the
second fundamental forms and prove Theorem \ref{them-main}.
Interior estimates are already proved in Theorem \ref{inregularity}.
Next, we estimate the higher order derivatives of $L,M,N$
in the geodesic coordinates as in (\ref{geod1}) and (\ref{geod2})
near the boundary. We need Lemma \ref{regular1} and Lemma
\ref{regular2} in the proof of the following result.

\begin{theorem}\label{mainregu}
Let $k\ge 2$ be an integer,
$\Sigma$ be an Alexandrov-Nirenberg surface in $\mathbb R^3$ of class $C^{k+6}$,
with the
induced metric $g$ in $D$, and ${\bf r}$ be the position vector of $\Sigma$.
Then,  for some $\alpha\in (0,1)$,
in the geodesic coordinates based a connected component of $\partial
D$ as in \eqref{geod1} and \eqref{geod2}, with $\Omega_t=[0,2\pi]\times (0,t)$,
\begin{equation*}|\nabla^k{M}|_{C^{\alpha}(\bar\Omega_{1/2})},
|\nabla^k{N}|_{C^{\alpha}(\bar\Omega_{1/2})}\\
\le C\left(\alpha, k, \, |g|_{C^{k+5}(\bar \Omega_1)},\,
\max_{t=0}\frac 1{|\nabla K|},\, \max_{t=0}\frac 1{k_g}\right).
\end{equation*}
\end{theorem}

\begin{proof} Let $\sigma$ be a connected component of $\partial\Sigma$ and take
the geodesic
coordinates based on $\sigma$ as in (\ref{geod1}) and
(\ref{geod2}). By Lemma \ref{Remark1.13},
Theorem \ref{moser3}, Theorem \ref{N2} and Theorem \ref{M2}, we
have
\begin{align}\label{709}
&\frac{1}{C^*}\le N\le C_*,\\
&|M(s,t)|+ |N(s,t)-N(s,0)|\le C_*t \quad\text {for any }t\in
[0,1],\label{723} \\
&\int_0^{\frac
12}\int_0^{2\pi}\left(t(\partial_tN)^2+(\partial_sN)^2\right)dsdt\le
C_* ,\label{724}\end{align} where $C_*$ is a positive constant
depending only on the quantities
in (\ref{maincond}). We now prove estimates of higher derivatives
near $(s,t)=(0,0)$.

We first rewrite the  equation (\ref{eq5/N}) for $1/N$. In view of
(\ref{GC3}), it is easy to see
\begin{align}\label{307}\begin{split}
NL&=t\left(\partial_tK(0)+\frac
{M^2}t+sc_1+tc_2\right)=t\bar a_{22},\\
N^2&=N^2(0)+(N^2-N^2(0)),\\
A_{12}&=3K_t(0)+(N^2-N^2(s,0))c_3+sc_4+Mc_5+tc_6,
\end{split}\end{align}
for some smooth functions $c_i$, $i=1, \cdots, 6$, of $s,t,M$ and
$N$. Dividing both sides of (\ref{eq5/N}) by $\bar a_{22}/N$
reduces it to the equation of $u=1/N$ in the form
\begin{equation}\label{502}
\mathcal{L}u=tu_{tt}-ta_{12}u_{st}+a_{11}u_{ss}+
b_{2}u_t+b_{1}u_s=f\quad\text{in } \mathbb R^2_+,
\end{equation}
where
\begin{align*}
&a_{12}=\frac {NM}{t\bar a_{22}},\quad a_{11}=\frac {N^2}{\bar a_{22}},
\quad
b_{1}=\frac {A_{11}}{\bar a_{22}},\\
&b_{2}=3+\frac 1{\bar a_{22}}\big((N-N(s,0))\bar
c_1+s\bar c_2+M\bar c_3+t\bar c_4\big),
\end{align*}
for some smooth functions $\bar c_i$, $i=1,...,4$, of $s,t, M$ and
$N$. It is easy to see that, for the equation (\ref{502}), all the
assumptions in Lemma \ref{regular1} are satisfied by the
hypotheses in the present theorem.  Therefore, we can conclude
that, for some cutoff function $\varphi_r$,
\begin{align*}&\|\varphi_rN^{-1}\|_{W^{1,6}(\mathbb R^2_+)}+
\|t\varphi_rN^{-1}\|_{W^{2,6}(\mathbb R^2_+)}\\
&\qquad+\|\varphi_r\partial_s^2N^{-1}\|_{L^{6}(\mathbb R^2_+)}+\|t^{\frac 12}
\varphi_r\partial_{st}N^{-1}\|_{L^{6}(\mathbb R^2_+)}\le C,\end{align*}
where $C$ is a positive constant
depending only on the quantities
in (\ref{maincond}).
We now record (\ref{sysNM1})
in the form
\begin{equation}
\begin{split}\label{309}
\partial_sM&=-\frac{L}{N}\partial_tN+\frac{2M}{N}\partial_sN
-\frac{B_t}{B}L+\frac{B_s}{B}M\\
&\qquad-BB_tN-\frac{2B_t}{B}\frac{M^2}{N}+\frac{1}{N}(B^2K)_t,\\
\partial_tM&=\partial_sN-\frac{B_t}{B}M.
\end{split}
\end{equation} Then,
$$\|\varphi_rM\|_{W^{1,6}(\mathbb R^2_+)}\le C.$$
By the Sobolev embedding in \cite{HH}(Lemma B.3), we have, for $\gamma=1-\frac 12-\frac 26=\frac 16$,
\begin{equation}\label{501}
|\varphi_r N|_{C^\gamma}+ |\varphi_r M|_{C^\gamma}+
|\varphi_r \partial_sN|_{C^\gamma}+
|\varphi_rt\partial_t(N^{-1})|_{C^\gamma}\le C.\end{equation} By \eqref{309},
we have $\partial_t M\in C^\gamma$.
Then $$
\frac{M^2}{t}=M\int_{0}^1\partial_2M(s,\theta t)d\theta\in C^\gamma
$$
In view of (\ref{307}), we get
$$|\varphi_ra_{ij}|_{C^\gamma}+|\varphi_rb_{i}|_{C^\gamma}
+|\varphi_rf|_{C^\gamma}\le C,$$ for some smaller $r$ and some
constant $C$ under control.
Thus all the assumptions in Lemma
\ref{regular2} are satisfied if we take $\alpha=\gamma$
and hence,
$$I_\gamma(\varphi_rN^{-1})\le C_1.$$ Then combining with
(\ref{309}) yields
\begin{align*}&|\varphi_rN^{-1}|_{\dot{C}^{1,\gamma}}
+|D(\varphi_rN^{-1})|_{\dot{C}^{\gamma}}+
|D(\varphi_rM)|_{\dot{C}^{\gamma}}\\
&\qquad+|\partial_s(\varphi_rN^{-1})|_{\dot{C}^{1,\gamma}}+
|t\partial_{t}(\varphi_rN^{-1})|_{\dot{C}^{1,\gamma}}\leq
C_1,\end{align*} for some constant $C_1$ and smaller $r=r_1$
depending only on $|g|_{C^5(\bar D)}$ and the quantities in
(\ref{maincond}). Next, we proceed by induction.  Assume, for some
$k\ge 1$ and $r=r_k>0$,
\begin{align}\label{710}\begin{split}
&|\varphi_ra_{ij}|_{\dot{C}^{k,\gamma}}+|\varphi_rN^{-1}|_{\dot{C}^{k,\gamma}}
+|D(\varphi_rN^{-1})|_{\dot{C}^{k-1,\gamma}}
+|D(\varphi_rM)|_{\dot{C}^{k-1,\gamma}}\\
&\qquad+|\partial_{s}(\varphi_rN^{-1})|_{\dot{C}^{k,\gamma}}+
|t\partial_{t}(\varphi_rN^{-1})|_{\dot{C}^{k,\gamma}}\leq C_k,
\end{split}\end{align}
where $C_k$ and $r_k$ are positive constants depending only on
$|g|_{C^{4+k}(\bar D)}$ and the quantities
in (\ref{maincond}).
Applying Lemma \ref{regular2} to (\ref{502}) for
$\alpha=k+\gamma$, we  get
$$|D(\varphi_rN^{-1})|_{\dot{C}^{k,\gamma}}+ |\partial_{ss}(\varphi_rN^{-1})|_{\dot{C}^{k,\gamma}}+
|t(\varphi_rN^{-1})|_{\dot{C}^{k+2,\gamma}}\leq C_{k+1},$$
where $r=r_{k+1}$ and $C_{k+1}$ are positive constants also depending on $
C^{4+k}$-norm of $g$. This implies, with (\ref{309}),
$$|t\partial_{t}(\varphi_rN^{-1})|_{\dot{C}^{k+1,\gamma}}
+|D(\varphi_rM)|_{\dot{C}^{k,\gamma}}\leq C_{k+1}. $$
Thus we have completed the proof of
(\ref{710}) for $k+1$.

Finally, differentiating (\ref{502}) and (\ref{309})
in $t$ and using Lemma \ref{regular2} we  get  estimates for
higher order derivatives of $N$ an $M$.
Thus the present theorem
has been proved.
\end{proof}

By combining Theorem \ref{inregularity} and Theorem \ref{mainregu}, we obtain
the following global estimate.

\begin{theorem}\label{mainTh}
Let $k\ge 2$ be an integer,
$\Sigma$ be an Alexandrov-Nirenberg surface in $\mathbb R^3$ of class $C^{k+4}$,
with the
induced metric $g$ in $D$, and ${\bf r}$ be the position vector of $\Sigma$.
Then, for any $\alpha\in (0,1)$,
\begin{equation*}|{\bf r}|_{C^{k,\alpha}(\bar D)}\\
\le C\left(k, \alpha, |g|_{C^{k+3}(\bar D)},\, \max_{\partial D}\frac
1{|\nabla K|},\, \max_{\partial D}\frac 1{|k_g|}\right).
\end{equation*}\end{theorem}

Theorem \ref{them-main} follows as a consequence of Theorem \ref{mainTh}.

\section{Appendix: $W^{2,p}$ Estimates and Schauder Estimates}
\label{sec-Appendix}

In this section, we prove
several regularity estimates for
degenerate elliptic equations as (\ref{eq5/N}) with characteristic degeneracy
on boundary. Most related techniques and notations are used in \cite{HH}.

For the sake of convenience,  we first give a brief explanation.
For $p\in (1,\infty)$ and
$\alpha\in (0,1)$, define $I_p(u)$ and $I_\alpha(u)$ by
$$I_p(u)=\|u\|_{L^p(\mathbb R^2_+)}+\|u_t\|_{L^p(\mathbb R^2_+)}+
\|u_{ss}\|_{L^p(\mathbb R^2_+)}+\|t^{\frac 12} u_{st}\|_{L^p(\mathbb R^2_+)}
+\|tu_{tt}\|_{L^p(\mathbb R^2_+)},
$$
and
$$I_\alpha(u)=\|u\|_{\dot{C}^{\alpha}(\bar{\mathbb R}^2_+)}+\|u_t\|_{\dot{C}^{\alpha}(\bar{\mathbb R}^2_+)}
+\|u_{ss}
\|_{\dot{C}^\alpha(\bar{\mathbb R}^2_+)} +\|t^{\frac
12}u_{st}\|_{\dot{C}^\alpha(\bar{\mathbb R}^2_+)}
+
\|tu_{tt}\|_{\dot{C}^\alpha(\bar{\mathbb R}^2_+)}.$$
For  an $\alpha$ in
$\mathbb{R}_+^1\backslash \mathbb{Z}$, we define a function
$f$ in $\dot{C}^\alpha(\bar {\mathbb R}^2_+)$  if
\begin{equation}\label{003}\|f\|_{\dot{C}^\alpha(\bar{\mathbb R}^2_+)}=
\sum_{|\beta|\le
[\alpha]}|D_x^{\beta}f|_{C(\bar{\mathbb R}^2_+)}+
[f]_{\dot{C}^{\alpha}(\bar{\mathbb R}^2_+)}<\infty,
\end{equation}
where
\begin{equation}[f]_{\dot{C}^\alpha(\bar{\mathbb R}^2_+)}
=\sum_{|\beta|=[\alpha]}\underset{ y\ge 0, x\neq\bar{x}\in
\mathbb{R}^1}{\sup}\left(\frac{| D_x^\beta f(x, y)-D_x^\beta
f(\bar{x}, y)|}{|x-\bar{x}|^{\alpha-[\alpha]}}\right).\end{equation}
It should be emphasized that the derivatives involved in $\dot C^{\alpha}$-norm are all $x$-directions.
Denote by
$\overline{W}^{2,p}$ as the completion of
$C_c^\infty(\bar{\mathbb R}^2_+)$ under the norm $I_p$.

For $a>3/2$, consider a
degenerate elliptic boundary value problem
\begin{equation}
\begin{split}\label{706}
&Lu=t\partial^2_{t}u+a\partial_tu+\partial_{ss}u=f
\quad\text {in }\mathbb R^2_+,\\
&u\rightarrow 0\text { as }s^2+t^2\rightarrow \infty\text
{ and }u \text { is bounded near }t=0.
\end{split}
\end{equation}
We recall a result about a special solution $u=K(f)$. (See  \cite{HH} for details.)

\begin{theorem}\label{LpHol}
Let $a>3/2$ be a constant  and let $p\in [2,  \infty)$ and $\alpha\in (0,1)$.
Then, for any $f\in C^{\infty}(\bar {\mathbb R}^2_+)$ with
$supp\{f\}\subset \{|s|\le T,0\le t\le T\}$,  $u=K(f)$ satisfies
\begin{align}\label{td2}\begin{split}&\|tu_{tt}\|_{L^p(\mathbb{R}^{2}_+)}+
\|t^\frac{1}{2}u_{st}\|_{L^p(\mathbb{R}^2_+)}+
\|u_{ss}\|_{L^p(\mathbb{R}^2_+)}+\|u_t\|_{L^p(\mathbb{R}^2_+)}+
\|u\|_{L^p(\mathbb{R}^2_+)}\\
&\qquad\le  C_{pT}\|f\|_{L^p(\mathbb{R}^2_+)},
\end{split}\end{align}
and
\begin{align}\label{td1}\begin{split}
& [tu_{tt}]_{\dot
C^\alpha(\bar{\mathbb R}^{2}_+)}+
[t^\frac{1}{2}u_{st}]_{\dot
C^\alpha(\bar{\mathbb R}^{2}_+)}+[u_{ss}]_{\dot
C^\alpha(\bar{\mathbb R}^{2}_+)}+[u_t]_{\dot
C^\alpha(\bar{\mathbb R}^{2}_+)}+|u|_{
C(\bar{\mathbb R}^{2}_+)}\\
&\qquad \le  C_{\alpha
T}|f|_{\dot
C^\alpha(\bar{\mathbb R}^{2}_+)},\end{split}\end{align}
for some universal constants  $C_{pT}$  and $C_{\alpha T}$
depending only on $n$, $a$ and $T$,  and $p$ and $\alpha$ respectively.
\end{theorem}

In the following, we study the regularity of solutions of
\begin{equation}\label{701}
\mathcal{L}u=tu_{tt}-ta_{12}u_{st}+a_{11}u_{ss}+b_{2}u_t+b_{1}u_s=f\text{
in } \mathbb R^2_+.
\end{equation}
Let $\varphi\in C_0^{\infty}(B_1)$ be a cut-off function and
$\varphi=1$ in $B_{1/4}$.  Define
$$\varphi_r(s,t)=\varphi\left(\frac sr,\frac tr\right).$$
Now we have two
lemmas.

\begin{lemma}\label{regular1}
Let $a_{12}, b_{1}$ be bounded and $a_{11}$ and $b_2$ be continuous near
the origin $0\in\mathbb R^2$ with
$a_{11}(0)=1$ and $b_{2}(0)>2$.
Suppose $u\in C^2({\mathbb R}^2_+)\cap L^{\infty}_{loc}(\bar{\mathbb R}^2_+)$,
with $tu_{t},u_s\in L^2_{loc}(\bar{\mathbb R}^2_+)$, satisfies \eqref{701},
for some  $f\in L^\infty_{loc}(\bar{\mathbb R}_+^2)$. Then, there exists an $r>0$ such that
$$\|D(\varphi_ru)\|_{L^6(\mathbb R^2_+)}+ \|\varphi_rtu\|_{W^{2,6}
(\mathbb R^2_+)}+\|t^{\frac 12}\varphi_r\partial_{st}u\|_{{L^6}
(\mathbb R^2_+)}+\|\varphi_r\partial^2_{s}u\|_{{L^6}
(\mathbb R^2_+)}+ |\varphi_ru|_{C^{\frac 23}(\bar {\mathbb R}^2_+)}\leq C,$$
where $C$ is a positive constant depending only on the $L^2$-norms of $\varphi_{2r}u$,
$\varphi_{2r}u_s$ and $t\varphi_{2r}u_t$,  the
modulus continuity  of $a_{11}$ and $b_{2}$ at $0$,
and the $L^{\infty}$-norms of $\varphi_{2r} a_{12}$, $\varphi_{2r}
b_{1}$ and $\varphi_{2r}f$.
\end{lemma}

\begin{proof} We write
\begin{equation}\label{702}
a_{11}=1+\bar a_{11},\, \, b_{2}=a +\bar b_{2} \text { with }\bar
a_{11}(0)=\bar b_{2}(0)=0, \, \, a>2,
\end{equation}
for some continuous functions $\bar a_{11}$ and $\bar b_{2}$ and some
constant $a$. Set $u_r=\varphi_ru$.  Then $u_r$
satisfies\begin{equation}\label{704}
\mathcal{L}_1u_r\equiv t\partial_{tt}u_r+\partial_{ss}u_r+a\partial_tu_r+Q(u_r)=f_r,
\end{equation}
where
$$Q(u_r)=\varphi_{2r}\left(\bar b_{2}\partial_tu_r
-ta_{12}\partial_{ts}u_r+\bar a_{11}\partial_{ss}u_r\right),$$
and
\begin{align*}f_r&=
(\mathcal{L}\varphi_r)u+2t\partial_t\varphi_ru_t-
2ta_{12}\partial_s\varphi_ru_t-2ta_{12}\partial_t\varphi_ru_s+2\partial_s\varphi_ru_s\\
&\qquad -b_{1}\partial_su_r+\varphi_rf.
\end{align*}
By the assumption
of the present lemma, it is easy to see $f_r\in L^2(\mathbb R^2_+)$.
For some $\lambda\in (0,1]$, change
the variables $s\rightarrow \lambda^{-1}s,t\rightarrow
\lambda^{-2}t,$ and still denote the new variables
by $s,t$. Then equation \eqref{704} is reduced to
\begin{equation}\label{705}
t\partial_{tt}u_r+\partial_{ss}u_r+a\partial_tu_r+Q_\lambda(u_r)=f_{r,\lambda}=\lambda^2f_r,
\end{equation}
where
$$Q_\lambda(u_r)=\varphi_{2r}\left(\bar b_{2}\partial_tu_r-\lambda ta_{12}
\partial_{st}u_r+\bar a_{11}\partial_{ss}u_r\right).$$
Using the operator $K$ in Theorem \ref{LpHol}, we can rewrite
(\ref{705}) in an  integral equation
\begin{equation}\label{707}
u_r=R(u_r)=\lambda^2K(f_{r})-K(Q_\lambda(u_r)).
\end{equation}
Set
$$I^*= I_2(K(f_{r})),$$
and
$$S_2=\{v\in \overline{W}^{2,2}:\, I_2(v)\leq I^*\}.$$
We note, by Theorem \ref{LpHol},
$$I^*\le C\|f_r\|_{L^2}.$$
By (\ref{td2}), we have, for any $v\in S_2$,
\begin{align*} I_2(R(v))&\leq
\lambda^2I_2(K(f_{r}))+I_2(K(Q_\lambda(v)))\\&\leq
\lambda^2I^*+\big(\sup|\bar b_{2}\varphi_{2r}|+\sup|\bar
a_{11}\varphi_{2r}|+\lambda\sup|\varphi_{2r}
a_{12}|\big)I_2(v)\\
&\leq \left(\lambda^2+ \frac 12\right)I^*\le I^*,
\end{align*} if  $\lambda$ and $r $ are chosen small enough. This
follows from the assumptions on the continuity of $\bar a_{11},\bar
b_{2}$ at $0$ and the boundedness of $a_{12}$. We also have
$$I_2(R(v_1)-R(v_2))=I_2(Q_\lambda(v_1)-Q_\lambda(v_2))\leq \frac 12I_2(v_1-v_2),$$
for some smaller $\lambda$ and $r$. Then, by the contraction mapping principle,
there exists a $v\in S_2$ such that
$$v=R(v)=-K(Q_\lambda(v))+\lambda^2K(f_r).$$
Pulling back to  the original
coordinates $(s,t)$, we get
$$
\mathcal{L}_1v=tv_{tt}+v_{ss}+av_t+Q(v)=f_r,$$
and
\begin{equation}\label{708}
I_2(v)\le
C\|f_r\|_{L^2},
\end{equation}
for some constant $C$ under control. Lemma 5.2 in
\cite{HH} yields
$$|v(s,t)|\le \begin{cases} Ct^{-a+\frac12}&\text{for }t\ge 4,\\
C|s|^{-1}&\text{for }|s|\ge 4,\end{cases}$$
and hence, for any $\delta>0$ and
$\epsilon>0$,
\begin{equation*}
\lim_{R\rightarrow\infty}\inf_{s^2+t^2=R^2}(u_r-v+\delta
t^{\epsilon+1-a})\geq 0.
\end{equation*}
Also by the definition of $I_2(v)$, we have $tv\in
H^2(\Omega_1)$, where $\Omega_1=\mathbb R^1\times (0,1)$, and hence
$tv\in C^\alpha(B_R(0)\cap\bar
\Omega_1)$, for any $\alpha\in (0,1)$, by the Sobolev embedding.
Fixing $\epsilon>0$ such that
$a>2+\epsilon$, we have, for any $\delta>0$,
\begin{align*}
\, &\inf_{|s|\le R,\, t\rightarrow 0}(u_r-v+\delta
t^{\epsilon+1-a})\\
\geq&\,  \inf_{|s|\le R,\, t\rightarrow
0}t^{\epsilon+1-a}\left(t^{a-1-\epsilon}u_r+\frac
\delta2\right)+\inf_{|s|\le R,\, t\rightarrow 0}\frac
1t\left(-tv+\frac \delta2 t^{\epsilon+2-a}\right)\geq 0.
\end{align*}
Note that
\begin{equation*}
\mathcal{L}_1(t^{\epsilon+1-a})=(\epsilon+1-a)(\epsilon+\varphi_{2r}\bar
b_{2}) t^{\epsilon-a}\leq 0,
\end{equation*}
for some smaller $ r $ independent of $\delta$.
Hence,
$$\mathcal{L}_1(u_r-v\pm\delta t^{\epsilon+1-a})\lessgtr0.$$
Then an
application of the maximum principle yields
$$|u_r-v|\leq \delta t^{\epsilon+1-a}.$$ Passing to the limit
$\delta\rightarrow 0$, we have $u_r=v$. Therefore from (\ref{708})
and the definition of $I_2(v)$ it follows that, for some constant
$C$ under control,
$$\int t(\partial_t\partial_su_r)^2+(\partial_{ss}u_r)^2\le C,$$
which implies  $\partial_su_r\in L^6(\mathbb R^2_+)$ by (\ref{weight1}).
So far we have proved $t\partial_tu_r, \partial_su_r\in
L^6(\mathbb R^2_+)$. Repeating the same arguments, we can
prove that $I_6(u_r)$ is bounded for a smaller $r$ and hence,
$\partial_tu_r,\partial_su_r\in L^6(\mathbb R^2_+)$. Using the Sobolev
embedding theorem, we have $u_r\in C^\alpha(\bar{\mathbb R}^2_+)$
with $\alpha=2/3$. This ends the proof of the present lemma.
\end{proof}

\begin{lemma}\label{regular2}
In addition to the hypotheses in Lemma \ref{regular1}, we assume,
for some $\alpha\in
\mathbb R^1_+\setminus \mathbb Z$,
$$a_{12}, a_{11},
b_{2},b_{1}\in \dot C^{\alpha}(\bar {\mathbb R}^2_+),
\quad  f\in
\dot{C}^\alpha_{loc}(\bar{\mathbb R}_+^2),$$
and
$$u,tu_{t},u_s\in C^2(\mathbb R^2_+)\cap
\dot{C}_{loc}^\alpha(\bar{\mathbb R}^2_+).$$ Then, there
exists an $r=r(\alpha)>0$ such that
$$I_\alpha(\varphi_ru)\leq C,$$
where $C$ is a positive constant depending only on $\alpha$,
the $\dot C^{\alpha}$-norms of
$\varphi_{2r}u$,
$\varphi_{2r}u_s$, $t\varphi_{2r}u_t$ and
$\varphi_{2r}f$, and the $\dot C^{\alpha}$-norms
of  $\varphi_{2r}a_{12}$, $\varphi_{2r} a_{11}$,  $\varphi_{2r}
b_{2}$ and $\varphi_{2r}b_{1}$.
\end{lemma}

The proof is similar to that of  Lemma
{\ref{regular1} and is omitted.

Now we prove a regularity result.

\begin{theorem}\label{rem1}
Let $a_{12}, a_{11}, b_{2}, b_{1}$ and $ f$ be $C^{\infty}$ in $\bar{\mathbb R}^2_+$
with  $a_{11}>0$ and $b_{2}(s,0)> 2$, and
$u$ be a solution of
\eqref{701} with $t^\frac 12\partial_t u,
\partial_s u\in L^2_{loc}(\bar{\mathbb R}^2_+)$. Then
$u$ is $C^\infty$ in $\bar{\mathbb R}^2_+$.
\end{theorem}

\begin{proof}
It suffices to prove the smoothness near $(s,t)=(0,0)$.
First, by Lemma \ref{regular1}, we have,
for some cutoff function $\varphi_r$,
$$\|\varphi_ru\|_{W^{1,6}(\mathbb R^2_+)}+
\|t\varphi_ru\|_{W^{2,6}(\mathbb R^2_+)}+\|\varphi_r\partial_s^2u\|_{L^6(\mathbb R^2_+)}+\|t^{\frac 12}\varphi_r\partial_{st}u\|_{L^6(\mathbb R^2_+)}\le C.$$
By the Sobolev embedding in \cite{HH}(Lemma B.3), we have, for $\gamma=1-\frac 12-\frac 26=\frac 16$,
\begin{equation}\label{501z}
|\varphi_r u|_{C^\gamma}+
|\varphi_rt\partial_tu|_{C^\gamma}+|\varphi_r\partial_su|_{C^\gamma}\le C.\end{equation}
Next, we apply Lemma
\ref{regular2} by taking $\alpha=\gamma$. Hence,
$$I_\gamma(\varphi_r u)\le C_1,$$ and, in particular,
$$|\varphi_ru|_{\dot{C}^{1,\gamma}}
+|D(\varphi_ru)|_{\dot{C}^{\gamma}}+|\partial_s(\varphi_ru)|_{\dot{C}^{1,\gamma}}+
|t\partial_{t}(\varphi_ru)|_{\dot{C}^{1,\gamma}}\leq C_1,$$
for some constant $C_1$ and smaller $r=r_1$.
Next, we proceed by induction.  Assume, for some $k\ge 1$ and
$r=r_k>0$,
\begin{equation}\label{710z}
|\varphi_ru|_{\dot{C}^{k,\gamma}}
+|D(\varphi_ru)|_{\dot{C}^{k-1,\gamma}}
+|\partial_{s}(\varphi_ru)|_{\dot{C}^{k,\gamma}}+
|t\partial_{t}(\varphi_ru)|_{\dot{C}^{k,\gamma}}\leq C_k.
\end{equation}
Applying Lemma \ref{regular2} to (\ref{701}) for
$\alpha=k+\gamma$, we  get
$$|D(\varphi_ru)|_{\dot{C}^{k,\gamma}}
+ |\partial_{ss}(\varphi_ru)|_{\dot{C}^{k,\gamma}}+
|t(\varphi_ru)|_{\dot{C}^{k+2,\gamma}}\leq C_{k+1},$$
where $r=r_{k+1}$ and $C_{k+1}$ are positive constants. This implies
$$|t\partial_{t}(\varphi_ru)|_{\dot{C}^{k+1,\gamma}}
\leq C_{k+1}. $$
Thus we have completed the proof of
(\ref{710z}) for $k+1$.

Finally, differentiating (\ref{701}) and using Lemma \ref{regular2},
we  get  estimates for
higher order derivatives of $u$. Thus the present theorem
has been proved.
\end{proof}


\begin{thebibliography}{99}

\bibitem{AD}
Adams, R.A., Fournier, J., \emph{Sobolev Spaces, 2nd Edition},
Academic Press, 2003.

\bibitem{AL}
Alexandrov, A.D., \emph{On a class of closed surface}, Recueil
Math. (Moscow), 4(1938), 69-77.

\bibitem{CY}
Cheng, S.-Y, Yau, S.-T., \emph{On the regularity of the Monge-Ampere
equation $\det(\partial^2u/\partial x_i\partial x_j)=F(x,u)$}, Comm. Pure Appl.Math.
30(1977), 41-68.

\bibitem{Cohn-Vossen1927}
Cohn-Vossen, S. E., \emph{Zwei S\"{a}tze \"{u}ber die Starrheit der
Eifl\"{a}chen}, Nach. Gesellschaft Wiss. G\"{o}ttinger, Math. Phys.
Kl.(1927), 125-134.



\bibitem{GL}
Guan, P.-F.,  Li, Y.-Y., \emph{The Weyl problem with nonnegative Gauss
curvature}, J. Diff. Geometry, 39(1994), 331-342.

\bibitem{GR}
Gromov, M., Rokhlin, V.A., \emph{Embeddings and immersions in Riemannian
geometry}, Russian Mathematical Surveys, 25(1970), 1-57.

\bibitem{GT}
Gilbarg, D., Trudinger, N.S., \emph{Elliptic Partial Differential
Equations of Second Order}, Springer-Verlag, Berlin, 2001.

\bibitem{Han2005} Han, Q., \emph{On the isometric embedding
of surfaces with Gauss curvature changing sign cleanly}, Comm.
Pure Appl. Math., 58(2005), 285-295.

\bibitem{Han2005b} Han, Q., \emph{Local isometric embedding of
surfaces with Gauss curvature changing sign stably across a curve},
Cal. Var. \& P.D.E., 25(2005), 79-103.

\bibitem{HQH}
Han, Q., Hong, J.-X., \emph{Isometric embedding of Riemannian manifolds
in Euclidean spaces}, American
Mathematical Society, Providence, RI, 2006.


\bibitem {Han-Hong-Lin2003} Han, Q., Hong, J.-X., Lin, C.-S.,
\emph{Local isometric embedding of surfaces with nonpositive
gaussian curvature}, J. Diff. Geometry, {63}(2003), 475-520.

\bibitem{Han&Khuri2010} Han, Q.,  Khuri, M.,
\emph{Local isometric embedding of surfaces with Gauss curvature of mixed sign},
Comm. Analysis \& Geometry, 18(2010), 649-704.

\bibitem{HL}
Han, Q., Lin, F.-H.,  \emph{On the isometric embedding of Torus in $\mathbb R^3$},
Methods Appl. Anal., 15(2008),  197-204.


\bibitem{Han-Lin2011} Han, Q., Lin, F.-H.,   \emph{Elliptic Partial Differential
Equations, 2nd Edition}, Courant Institute Lecture Notes, Volume 1, American
Mathematical Society, Providence, RI, 2011.

\bibitem{H}
Heinz, E., \emph{Neue a priori Abschatsungen fur den ortsvektor iener
Flache positiver Gaussscher Krummung durch ihr Linienelment}, Math.
Z., 74(1960), 1-52.

\bibitem{Herglotz1943} Herglotz, G., \emph{\"{U}ber die Starrheit der
Eifl\"{a}chen}, Abh. Math. Sem. Hansischen Univ., 15(1943),
127-129.


\bibitem{H1}
Hong, J.-X, \emph{Darboux equations and isometric embedding of Riemannian
manifolds with nonnegative curvature in $\mathbb R^3$}, Chin. Ann. of
Math., 20B(1999), 123-136.





\bibitem{HH}
Hong, J.-X., Huang, G.-G., \emph{$L^p $ and H\"older estimates for a class
of degenerate elliptic partial differential equations and its
applications}, Inter.Math.R.N, Vol. 2012, No.13, pp.2889-2941.


\bibitem{HHW}
Hong, J.-X., Huang, G.-G., Wang, W.-Y., \emph{Existence of Global smooth
solutions to Dirichlet problem for degenerate elliptic
Monge-Amp\`{e}re equations}, Comm.  in Partial Diff.
Equations, 36(2011), no.4,635-656.

\bibitem{HZ}
Hong, J.-X., Zuily, C., \emph{Isometric embedding of the 2-sphere with
nonnegative curvature in $\mathbb R^3$}, Math. Z., 219(1995), 323-334.

\bibitem {Lin1985} Lin, C.-S.,
\emph{The local isometric embedding in $\mathbb R^3$ of
2-dimensional Riemannian manifolds with nonnegative curvature}, J.
Diff. Geometry, 21(1985), 213-230.

\bibitem {Lin1986} Lin, C.-S., \emph{The local
isometric embedding in $\mathbb R^3$ of two dimensional Riemannian
manifolds with Gaussian curvature changing sign cleanly}, Comm.
Pure Appl. Math., 39(1986), 867-887.


\bibitem{L2}
Li, C.-H.,  \emph{The analyticity of solutions to a class of degenerate
elliptic equations},  Sci. China Math., 53(2010), 2061-2068.

\bibitem{K}
Keldys, M.V., \emph{On certain cases of degeneration of equations of
elliptic type on the boundary of a domain}, Dokl Akad. Nauk SSSR,
77(1951), 181-183.

\bibitem{N1}
Nirenberg, L., \emph{The Weyl and Minkowski problems in differential
geometry in the large},  Comm. Pure Appl. Math., 6(1953), 337-394.

\bibitem{N2}
Nirenberg, L., \emph{Rigidity of a class of closed surfaces}, edited by
Lang, 177-193, University of Wisconsin Press, 1963.



\bibitem{Pogorelov1953} Pogorelov, A.V., \emph{Regularity of a convex
surface with given Gaussian curvature} (Russian), Mat. Sbornik
(N.S.), 31(1952), 88-103.


\bibitem{Sacksteder1962} Sacksteder, R., \emph{The rigidity of hypersurfaces},
J. Math. Mech., 11(1962), 929-939.


\bibitem{Weyl1916}
Weyl, H., \emph{\"Uber die Bestimmheit einer geschlossenen konvex
Fl\"ache durch ihr Linienelement}, Vierteljahresschrift der
nat.-Forsch. Ges. Z\"urich, 61(1916), 40-72.

\bibitem{Y1}
Yau, S.-T., \emph{Problem section}, Seminar on Differential Geometry,
Princeton University Press, 1982.

\bibitem{Y2}
Yau, S.-T., \emph{Lecture on Differential Geometry}, in Berkeley, 1977.

\bibitem{Y3}
Yau, S.-T., \emph{Survey on geometry and analysis}, Asian J. Math.,
4(2000), 235-278.

\end{thebibliography}
\end{document}